\documentclass[fleqn,preprint,3p,a4paper]{elsarticle}

\usepackage{amssymb}
\usepackage{amsmath}
\usepackage{amsthm}
\usepackage{dcolumn}
\usepackage{endnotes}
\usepackage{tabularx}
\usepackage[matrix,arrow]{xy}
\usepackage{wasysym}

\newtheorem{theorem}{Theorem}[section]
\newtheorem{proposition}[theorem]{Proposition}
\newtheorem{lemma}[theorem]{Lemma}
\newtheorem{corollary}[theorem]{Corollary}

\theoremstyle{definition}
\newtheorem{definition}[theorem]{Definition}
\newtheorem{example}[theorem]{Example}
\newtheorem{remark}[theorem]{Remark}
\newtheorem{question}[theorem]{Question}

\newcommand{\ir}{{\mathsf{Irr}}}

\newcommand{\cl}{{\rm cl}}
\newcommand{\ii}{{\rm int}}

\newcommand{\ua}{\mathord{\uparrow}}
\newcommand{\da}{\mathord{\downarrow}}

\newcommand{\mk}{\mathord{\mathsf{K}}}
\newcommand{\wdd}{\mathord{\mathsf{WD}}}
\newcommand{\md}{\mathord{\mathsf{D}}}
\newcommand{\dc}{\mathord{\mathsf{DC}}}
\newcommand{\kf}{\mathord{\mathsf{RD}}}

\journal{Topology and its applications}

\begin{document}

\begin{frontmatter}



\title{On $T_0$ spaces determined by well-filtered spaces\tnoteref{t1}}
\tnotetext[t1]{This research was supported by the National Natural Science Foundation of China (Nos. 11661057, 11361028,
61300153, 11671008, 11701500, 11626207); the Natural Science Foundation of Jiangxi Province , China (No. 20192ACBL20045); NSF Project of Jiangsu Province,
China (BK20170483); and NIE ACRF (RI 3/16 ZDS), Singapore}

\author[X. Xu]{Xiaoquan Xu\corref{mycorrespondingauthor}}
\cortext[mycorrespondingauthor]{Corresponding author}
\ead{xiqxu2002@163.com}
\address[X. Xu]{School of Mathematics and Statistics,
Minnan Normal University, Zhangzhou 363000, China}
\author[C. Shen]{Chong Shen}
\address[C. Shen]{School of Mathematics and Statistics,
Beijing Institute of Technology, Beijing 100081, China}
\ead{shenchong0520@163.com}
\author[X. Xi]{Xiaoyong Xi}
\ead{littlebrook@jsnu.edu.cn}
\address[X. Xi]{School of mathematics and Statistics,
Jiangsu Normal University, Xuzhou 221116, China}
\author[D. Zhao]{Dongsheng Zhao}
\address[D. Zhao]{Mathematics and Mathematics Education,
National Institute of Education Singapore, \\
Nanyang Technological University,
1 Nanyang Walk, Singapore 637616}
\ead{dongsheng.zhao@nie.edu.sg}

\begin{abstract}
  We first introduce and study two new classes of subsets in $T_0$ spaces - Rudin sets and $\wdd$ sets lying between the class of all closures of directed subsets and that of irreducible closed subsets. Using such subsets, we define three new types of topological  spaces - $\mathsf{DC}$ spaces, Rudin spaces and $\wdd$ spaces. The class of Rudin spaces lie between the class of  $\wdd$ spaces and that of  $\dc$ spaces, while the class of  $\dc$ spaces lies  between the class of Rudin spaces and that of sober spaces. Using Rudin sets and $\wdd$ sets, we formulate and prove a number of  new characterizations of well-filtered spaces and sober spaces. For a $T_0$ space $X$, it is proved that $X$ is sober if{}f $X$ is a well-filtered Rudin space if{}f $X$ is a well-filtered $\mathsf{WD}$ space. We also prove that every  locally compact $T_0$ space is a  Rudin space, and every core compact $T_0$ space is a $\wdd$ space. One immediate corollary is that  every core compact well-filtered space is sober, giving a positive answer to Jia-Jung problem. Using $\wdd$ sets, we present a more directed construction of the well-filtered reflections of $T_0$ spaces, and prove that the products of any collection of well-filtered spaces are well-filtered. Our study also leads to a number of problems, whose answering will deepen our understanding of the related spaces and structures.
\end{abstract}





\end{frontmatter}


\section{Introduction}

In the theory of non-Hausdorff topological spaces, the $d$-spaces, well-filtered spaces and sober spaces form three of the most important classes. Rudin's Lemma is a useful tool in topology and plays a crucial role in domain theory (see [2-9, 12-30]). In recent years, it has been used to study the various aspects of well-filtered spaces and sober spaces, initiated by Heckmann and Keimel
\cite{Klause-Heckmann}. In this paper, inspired by the topological version of Rudin's Lemma by Heckmann and Keimel, Xi and Lawson's work \cite{Xi-Lawson-2017} on well-filtered spaces and our recent works \cite{Shenchon, xuxizhao} on sober spaces and well-filtered reflections of $T_0$ spaces, we introduce and investigate two new classes of subsets in $T_0$ spaces - Rudin sets and $\wdd$ sets lying between the class of all closures of directed subsets and that of irreducible closed subsets. Using such subsets, we introduce and study three new types of topological spaces - directed closure spaces ($\mathsf{DC}$ spaces for short), Rudin spaces and well-filtered determined spaces ($\wdd$ spaces for short). Rudin spaces lie between $\wdd$ spaces and $\dc$ spaces, and $\dc$ spaces lie between Rudin spaces and sober spaces. We shall prove  that closed subspaces, retracts and products of Rudin spaces (resp. $\wdd$ spaces) are again Rudin spaces (resp., $\wdd$ spaces). Using Rudin sets and $\wdd$ sets, we formulate and prove a number of  new characterizations of well-filtered spaces and sober spaces. For a $T_0$ space $X$, it is proved that $X$ is sober if{}f  $X$ is a well-filtered Rudin space if{}f $X$ is a well-filtered $\wdd$ space. In \cite{E_20182}, Ern\'e proved that in a locally hypercompact $T_0$ space $X$, every irreducible closed subset $A$ of $X$ is the closure of a directed subset of $X$. So locally hypercompact $T_0$ spaces are $\mathsf{DC}$ spaces. Furthermore, we prove that every locally compact $T_0$ space is a  Rudin space and every core compact $T_0$ space is a $\wdd$ space. As a corollary we deduce that every core compact well-filtered space is sober, giving a positive answer to Jia-Jung problem \cite{jia-2018}, which has been independently answered by Lawson and Xi \cite{Lawson-Xi} using a different method.

It is well-known that the category of all sober spaces ($d$-spaces) is reflective in the category of all $T_0$ spaces (see [8, 13, 24-26]). But for quite a long time,
it is not known whether the category of all well-filtered spaces is reflective in the category of all $T_0$ space.
Recently, following Keimel and Lawson's method \cite{Keimel-Lawson}, which originated from Wyler's method \cite{Wyler}, Wu, Xi, Xu and Zhao [9] gave a positive answer to the above problem. Following Ershov's method of constructing the $d$-completion of $T_0$ spaces, Shen, Xi, Xu and Zhao presented a different construction of the well-filtered reflection of $T_0$ spaces. In the current paper, using $\wdd$ sets, we present a more direct construction of the well-filtered reflections of $T_0$ spaces, and prove that products of well-filtered spaces are well-filtered. Some major properties of well-filtered reflections of $T_0$ spaces are investigated. Comparatively, the technique presented in this paper is not just more direct, but also more simple. In addition, it can be directly applied to the general $K$-ifications in the sense of Keimel and Lawson \cite{Keimel-Lawson}. In
a forthcoming article we will use the technique to set up the $K$-ification theory
of $T_0$ spaces. Our study also leads to a number of problems, whose answering will deepen our understanding of the related spaces and structures.

\section{Preliminary}

In this section, we briefly  recall some fundamental concepts and notations that will be used in the paper. Some basic properties of irreducible sets and compact saturated sets are presented.

For a poset $P$ and $A\subseteq P$, let
$\mathord{\downarrow}A=\{x\in P: x\leq  a \mbox{ for some }
a\in A\}$ and $\mathord{\uparrow}A=\{x\in P: x\geq  a \mbox{
	for some } a\in A\}$. For  $x\in P$, we write
$\mathord{\downarrow}x$ for $\mathord{\downarrow}\{x\}$ and
$\mathord{\uparrow}x$ for $\mathord{\uparrow}\{x\}$. A subset $A$
is called a \emph{lower set} (resp., an \emph{upper set}) if
$A=\mathord{\downarrow}A$ (resp., $A=\mathord{\uparrow}A$). Define $A^{\uparrow}=\{x\in P : x \mbox{ is an upper bound of } A \mbox{ in }P\}$. Dually, define $A^{\downarrow}=\{x\in P : x \mbox{ is a lower bound of } $A$ \mbox{ in } $P$\}$. The set $A^{\delta}=(A^{\ua})^{\da}$ is called the \emph{cut} generated by $A$. Let $P^{(<\omega)}=\{F\subseteq P : F \mbox{~is a nonempty finite set}\}$ and $\mathbf{Fin} ~P=\{\uparrow F : F\in P^{(<\omega)}\}$.
 For a nonempty subset $A$ of $P$, define $max (A)=\{a\in A : a \mbox{~ is a maximal element of~} A\}$ and $min (A)=\{a\in A : a \mbox{~ is a minimal element of~} A\}$.

A nonempty subset $D$ of a poset $P$ is \emph{directed} if every two
elements in $D$ have an upper bound in $D$. The set of all directed sets of $P$ is denoted by $\mathcal D(P)$. $I\subseteq P$ is called an \emph{ideal} of $P$ if $I$ is a directed lower subset of $P$. Let $\mathrm{Id} (P)$ be the poset (with the order of set inclusion) of all ideals of $P$. Dually, we define the concept of \emph{filters} and denote the poset of all filters of $P$ by $\mathrm{Filt}(P)$.  $P$ is called a
\emph{directed complete poset}, or \emph{dcpo} for short, if for any
$D\in \mathcal D(P)$, $\bigvee D$ exists in $P$. $P$ is said to be \emph{Noetherian} if it satisfies the \emph{ascending chain condition} ($\mathrm{ACC}$ for short): every ascending chain has a greatest member. Clearly, $P$ is Noetherian if{}f every directed set of $P$ has a largest element (equivalently, every ideal of $P$ is principal).

As in \cite{redbook}, the \emph{upper topology} on a poset $Q$, generated
by the complements of the principal ideals of $Q$, is denoted by $\upsilon (Q)$. A subset $U$ of $Q$ is \emph{Scott open} if
(i) $U=\mathord{\uparrow}U$ and (ii) for any directed subset $D$ for
which $\bigvee D$ exists, $\bigvee D\in U$ implies $D\cap
U\neq\emptyset$. All Scott open subsets of $Q$ form a topology,
and we call this topology  the \emph{Scott topology} on $P$ and
denote it by $\sigma(P)$. The space $\Sigma~\!\! Q=(Q,\sigma(Q))$ is called the
\emph{Scott space} of $Q$. The upper sets of $Q$ form the (\emph{upper}) \emph{Alexandroff topology} $\alpha (Q)$.

The category of all $T_0$ spaces is denoted by $\mathbf{Top}_0$. For a subcategory  $\mathbf{K}$ of the category $\mathbf{Top}_0$, the objects of $\mathbf{K}$ will be called $\mathbf{K}$-spaces. For $X\in \mathbf{Top}_0$, we use $\leq_X$ to represent the \emph{specialization order} of $X$, that is, $x\leq_X y$ if{}f $x\in \overline{\{y\}}$). In the following, when a $T_0$ space $X$ is considered as a poset, the order always refers to the specialization order if no other explanation. Let $\mathcal O(X)$ (resp., $\mathcal C(X)$) be the set of all open subsets (resp., closed subsets) of $X$, and let $\mathcal S^u(X)=\{\ua x : x\in X\}$. Define $\mathcal S_c(X)=\{\overline{{\{x\}}} : x\in X\}$ and $\mathcal D_c(X)=\{\overline{D} : D\in \mathcal D(X)\}$.

\begin{remark}\label{Adelta=clAdelta} Let $X$ be a $T_0$ space, $C\subseteq X$ and $x\in X$. Then the followings are equivalent:
\begin{enumerate}[\rm (1)]
   \item $x\in C^{\ua}$;
   \item $C\subseteq \da x$;
   \item $\overline{C}\subseteq \da x$;
   \item $x\in \overline{C}^{\ua}$.
\end{enumerate}
Therefore,

$\bigcap_{c\in C}\ua c=C^{\ua}=\overline{C}^{\ua}=\bigcap_{b\in \overline{C}}\ua b $ ~~~
and~~~
$ C^{\delta}=\bigcap \{\da x : C\subseteq \da x\}=\bigcap \{\da x : \overline{C}\subseteq \da x\}=\overline{C}^{\delta}.$
\end{remark}

For a $T_0$ space $X$ and a nonempty subset $A$ of $X$, $A$ is \emph{irreducible} if for any $\{F_1, F_2\}\subseteq \mathcal C(X)$, $A \subseteq F_1\cup F_2$ implies $A \subseteq F_1$ or $A \subseteq  F_2$.  Denote by $\ir(X)$ (resp., $\ir_c(X)$) the set of all irreducible (resp., irreducible closed) subsets of $X$. Clearly, every subset of $X$ that is directed under $\leq_X$ is irreducible. $X$ is called \emph{sober}, if for any  $F\in\ir_c(X)$, there is a unique point $a\in X$ such that $F=\overline{\{a\}}$. The category of all sober spaces with continuous mappings is denoted by $\mathbf{Sob}$.

The following two lemmas on irreducible sets are well-known.

\begin{lemma}\label{irrsubspace}
Let $X$ be a space and $Y$ a subspace of $X$. Then the following conditions are equivalent for a
subset $A\subseteq Y$:
\begin{enumerate}[\rm (1)]
	\item $A$ is an irreducible subset of $Y$.
	\item $A$ is an irreducible subset of $X$.
	\item ${\rm cl}_X A$ is an irreducible closed subset of $X$.
\end{enumerate}
\end{lemma}

\begin{lemma}\label{irrimage}
	If $f : X \longrightarrow Y$ is continuous and $A\in\ir (X)$, then $f(A)\in \ir (Y)$.
\end{lemma}

\begin{remark}\label{subspaceirr}  If $Y$ is a subspace of a space $X$ and $A\subseteq Y$, then by Lemma \ref{irrsubspace}, $\ir (Y)=\{B\in \ir(X) : B\subseteq Y\}\subseteq \ir (X)$ and  $\ir_c (Y)=\{B\in \ir(X) : B\in \mathcal C(Y)\}\subseteq \ir (X)$. If $Y\in \mathcal C(X)$, then $\ir_c(Y)=\{C\in \ir_c(X) : C\subseteq Y\}\subseteq \ir_c (X)$.
\end{remark}

\begin{lemma}\label{irrprod}\emph{(\cite{Shenchon})}
	Let	$X=\prod_{i\in I}X_i$ be the product space of $T_0$ spaces  $X_i (i\in I)$. If  $A$ is an irreducible subset of $X$, then $\cl_X(A)=\prod_{i\in I}\cl_{X_i}(p_i(A))$, where $p_i : X \longrightarrow X_i$ is the $i$th projection for each $i\in I$.
\end{lemma}

\begin{lemma}\label{prodirr}
	Let	$X=\prod_{i\in I}X_i$ be the product space of $T_0$ spaces  $X_i (i\in I)$  and $A_i\subseteq X_i$ for each $i\in I$. Then the following two conditions are equivalent:
\begin{enumerate}[\rm (1)]
	\item $\prod_{i\in I}A_i\in \ir (X)$.
	\item $A_i\in \ir (X_i)$ for each $i\in I$.
\end{enumerate}
\end{lemma}
\begin{proof}  (1) $\Rightarrow$ (2): By Lemma \ref{irrimage}.

(2) $\Rightarrow$ (1): Let $A=\prod_{i\in I}A_i$. For $U, V\in \mathcal O(X)$, if $A\cap U\neq\emptyset\neq A\cap V$, then there exist $I_1, I_2\in I^{(<\omega)}$ and $(U_i, V_j)\in \mathcal O(X_i)\times \mathcal O(X_j)$ for all $(i, j)\in I_1\times I_2$ such that $\bigcap_{i\in I_1}p_i^{-1}(U_i)\subseteq U$, $\bigcap_{j\in I_2}p_j^{-1}(V_j)\subseteq V$ and $A\cap \bigcap_{i\in I_1}p_i^{-1}(U_i)\neq\emptyset\neq A\cap \bigcap_{j\in I_2}p_i^{-1}(V_j)$. Let $I_3=I_1\cup I_2$. Then $I_3$ is finite. For $i\in I_3\setminus I_1$ and $j\in I_3\setminus I_2$, let $U_i=X_i$ and $V_j=X_j$. Then for each $i\in I_3$, we have $A_i\cap U_i\neq\emptyset\neq A_i\cap V_i$, and whence $A_i\cap U_i\cap V_i\neq\emptyset$ by $A_i\in \ir (X_i)$. It follows that $A\cap \bigcap_{i\in I_1}p_i^{-1}(U_i)\cap \bigcap_{j\in I_2}p_i^{-1}(V_j)\neq\emptyset$, and consequently, $A\cap U\cap V\neq \emptyset$. Thus $A\in \ir (X)$.

\end{proof}

By Lemma \ref{irrprod} and Lemma \ref{prodirr}, we obtain the following corollary.

\begin{corollary}\label{irrcprod} Let $X=\prod_{i\in I}X_i$ be the product space of $T_0$ spaces  $X_i (i\in I)$. If  $A\in \ir_c(X)$, then $A=\prod_{i\in I}p_i(A)$ and $p_i(A)\in \ir_c (X_i)$ for each $i\in I$.
\end{corollary}

A $T_0$ space $X$ is called \emph{irreducible complete}, $r$-\emph{complete} for short, if for any
$A\in \ir(X)$, $\bigvee A$ exists in $X$. For a subset $B$ of $X$, $\bigvee B$ exists in $X$ if{}f $\bigvee \overline{B}$ exists in $X$, and $\bigvee B =\bigvee \overline{B}$ if they exist in $X$. So $X$ is irreducible complete if{}f $\bigvee A$ exists in $X$ for all
$A\in \ir_c(X)$.

\begin{remark}\label{Soberirrcomp} Every sober space is irreducible complete. In fact, if $X$ is a sober space and $A\in \ir (X)$, then there is an $x\in X$ such that $\overline{A}=\overline{\{x\}}$, and hence $\bigvee A=\bigvee\overline{A}=\bigvee \overline{\{x\}}=x$.
\end{remark}

Let $L$ be the complete lattice constructed by Isbell \cite{isbell}. Then $\Sigma L$ is irreducible complete, but is non-sober.

\begin{proposition}\label{uppertopSober} For any  poset $P$, the space $(P,\upsilon (P))$ is sober if{}f it is irreducible complete, where $\upsilon(P)$ is the upper topology on $P$.
\end{proposition}
\begin{proof} If the upper topology $\upsilon (P)$ is sober, then $(P, \upsilon (P))$ is irreducible complete by Remark \ref{Soberirrcomp}. Conversely, if $(P, \upsilon (P))$ is irreducible complete, we show that $\upsilon (P)$ is sober. For $A\in \ir ((P, \upsilon (P)))$, if $\cl_{\upsilon (P)} A=P$, then $P$ is irreducible in $(P, \upsilon (P))$ and hence has a largest element $\top$ since $(P, \upsilon (P))$ is irreducible complete. So $P=\downarrow \top=\overline{\{\top\}}$. If $\cl_{\upsilon (P)}A\neq P$, then there is a nonempty family $\{\da F_i : i\in I\}\subseteq \mathbf{Fin} (P)$ such that $\cl_{\upsilon (P)} A=\bigcap_{i\in I} \da F_i$. For each $i\in I$, $\cl_{\upsilon (P)} A\subseteq \da F_i$, and hence $\cl_{\upsilon (P)}A\subseteq \da x_i$ for some $x_i\in F_i$ by the irreducibility of $A$. Therefore, $\cl_{\upsilon (P)} A=\bigcap_{i\in I} \da x_i\supseteq\bigcap \{\da x : A\subseteq \da x\}=A^{\delta}\supseteq \cl_{\upsilon (P)} {A}$. Since $(P, \upsilon (P))$ is irreducible complete, $\bigvee A$ exists in $P$, and consequently, $\cl_{\upsilon (P)}A=A^{\delta}=\da \bigvee A=\overline{\{\bigvee A\}}$. Thus $\upsilon (P)$ is sober.
\end{proof}

For any topological space $X$, $\mathcal G\subseteq 2^{X}$ and $A\subseteq X$, let $\Diamond_{\mathcal G} A=\{G\in \mathcal G : G\bigcap A\neq\emptyset\}$ and $\Box_{\mathcal G} A=\{G\in \mathcal G : G\subseteq  A\}$. The symbols $\Diamond_{\mathcal G} A$ and $\Box_{\mathcal G} A$ will be simply written as $\Diamond A$  and $\Box A$ respectively if there is no confusion. The \emph{lower Vietoris topology} on $\mathcal{G}$ is the topology that has $\{\Diamond U : U\in \mathcal O(X)\}$ as a subbase, and the resulting space is denoted by $P_H(\mathcal{G})$. If $\mathcal{G}\subseteq \ir (X)$, then $\{\Diamond_{\mathcal{G}} U : U\in \mathcal O(X)\}$ is a topology on $\mathcal{G}$. The space $P_H(\mathcal{C}(X)\setminus \{\emptyset\})$ is called the \emph{Hoare power space} or \emph{lower space} of $X$ and is denoted by $P_H(X)$ for short (cf. \cite{Schalk}). Clearly, $P_H(X)=(\mathcal{C}(X)\setminus \{\emptyset\}, \upsilon(\mathcal{C}(X)\setminus \{\emptyset\}))$. So $P_H(X)$ is always sober by Proposition \ref{uppertopSober} (or \cite[Corollary 4.10]{ZhaoHo}). The \emph{upper Vietoris topology} on $\mathcal{G}$ is the topology that has $\{\Box_{\mathcal{G}} U : U\in \mathcal O(X)\}$ as a base, and the resulting space is denoted by $P_S(\mathcal{G})$.

\begin{remark} \label{eta continuous} Let $X$ be a $T_0$ space.
\begin{enumerate}[\rm (1)]
	\item If $\mathcal{S}_c(X)\subseteq \mathcal{G}$, then the specialization order on $P_H(\mathcal{G})$ is the order of set inclusion, and the \emph{canonical mapping} $\eta_{X}: X\longrightarrow P_H(\mathcal{G})$, given by $\eta_X(x)=\overline {\{x\}}$, is an order and topological embedding (cf. \cite{redbook, Jean-2013, Schalk}).
    \item The space $X^s=P_H(\ir_c(X))$ with the canonical mapping $\eta_{X}: X\longrightarrow X^s$ is the \emph{sobrification} of $X$ (cf. \cite{redbook, Jean-2013}).
\end{enumerate}
\end{remark}

For a space $X$, a subset $A$ of $X$ is called \emph{saturated} if $A$ equals the intersection of all open sets containing it (equivalently, $A$ is an upper set in the specialization order). We shall use $\mathord{\mathsf{K}}(X)$ to
denote the set of all nonempty compact saturated subsets of $X$ and endow it with the \emph{Smyth preorder}, that is, for $K_1,K_2\in \mathord{\mathsf{K}}(X)$, $K_1\sqsubseteq K_2$ if{}f $K_2\subseteq K_1$. $X$ is called \emph{well-filtered} if it is $T_0$, and for any open set $U$ and filtered family $\mathcal{K}\subseteq \mathord{\mathsf{K}}(X)$, $\bigcap\mathcal{K}{\subseteq} U$ implies $K{\subseteq} U$ for some $K{\in}\mathcal{K}$. The category of all well-filtered spaces with continuous mappings is denoted by $\mathbf{Top}_w$.
The space $P_S(\mathord{\mathsf{K}}(X))$, denoted shortly by $P_S(X)$, is called the \emph{Smyth power space} or \emph{upper space} of $X$ (cf. \cite{Heckmann, Schalk}). It is easy to see that the specialization order on $P_S(X)$ is the Smyth order (that is, $\leq_{P_S(X)}=\sqsubseteq$). The \emph{canonical mapping} $\xi_X: X\longrightarrow P_S(X)$, $x\mapsto\ua x$, is an order and topological embedding (cf. \cite{Heckmann, Klause-Heckmann, Schalk}). Clearly, $P_S(\mathcal S^u(X))$ is a subspace of $P_S(X)$ and $X$ is homeomorphic to $P_S(\mathcal S^u(X))$.

 \begin{lemma}\label{X-Smyth-irr} Let $X$ be a $T_0$ space and $A\subseteq X$. Then the following three conditions are equivalent:
 \begin{enumerate}[\rm (1)]
	\item $A\in\ir (X)$.
	\item $\xi_X(A)\in \ir (P_S(X))$.
    \item $\xi_X(A)\in \ir (P_S(\mathcal S^u(X)))$.
\end{enumerate}

Moreover, the following two conditions are equivalent:
 \begin{enumerate}[\rm (a)]
	\item $A\in\ir_c (X)$.
	\item $\xi_X(A)\in \ir_c (P_S(\mathcal S^u(X)))$.
\end{enumerate}
\end{lemma}
\begin{proof} (1) $\Rightarrow$ (2): By Lemma \ref{irrimage}.

(2) $\Rightarrow$ (3): By Remark \ref{subspaceirr} and $P_S(\mathcal S^u(X)))$ is a subspace of $P_S(X)$.

(3) $\Rightarrow$ (1) and (a) $\Leftrightarrow$ (b): Since $x\mapsto \ua x : X \longrightarrow P_S(\mathcal S^u(X))$ is a homeomorphism.
\end{proof}

\begin{remark}\label{meet-in-Smyth} Let $X$ be a $T_0$ space and $\mathcal A\subseteq \mk (X)$. Then $\bigcap \mathcal A=\bigcap \overline{\mathcal A}$, here the closure of $\mathcal A$ is taken in $P_S(X)$. Clearly, $\bigcap \overline{\mathcal A}\subseteq\bigcap \mathcal A$. On the other hand, for any $K\in \overline{\mathcal A}$ and $U\in \mathcal O(X)$ with $K\subseteq U$ (that is, $K\in \Box U$), we have $\mathcal A\bigcap\Box U\neq\emptyset$, and hence there is a $K_U\in \mathcal A\bigcap\Box U$. Therefore, $K=\bigcap \{U\in \mathcal O(X) : K\subseteq U\}\supseteq\bigcap \{K_U : U\in \mathcal O(X) \mbox{ and } K\subseteq U\}\supseteq\bigcap \mathcal A$. It follows that $\bigcap \overline{\mathcal A}\supseteq\bigcap \mathcal A$. Thus $\bigcap \mathcal A=\bigcap \overline{\mathcal A}$.
\end{remark}

\begin{lemma}\label{K union} \emph{(\cite{jia-Jung-2016, Schalk})}  Let $X$ be a $T_0$ space. If $\mathcal K\in\mk(P_S(X))$, then  $\bigcup \mathcal K\in\mk(X)$.
\end{lemma}

\begin{corollary}\label{Smythunioncont} \emph{(\cite{Schalk, jia-Jung-2016})}  For any $T_0$ space $X$ , the mapping $\bigcup : P_S(P_S(X)) \longrightarrow P_S(X)$, $\mathcal K\mapsto \bigcup \mathcal K$, is continuous.
\end{corollary}
\begin{proof} For $\mathcal K\in\mk(P_S(X))$, $\bigcup \mathcal K=\bigcup \mathcal K\in\mk(X)$ by Lemma \ref{K union}. For $U\in \mathcal O(X)$, we have $\bigcup^{-1}(\Box U)=\{\mathcal K\in \mk (P_S(X)) : \bigcup \mathcal K\in \Box U\}=\{\mathcal K\in \mk (P_S(X)) : \mathcal K\subseteq \Box U\}=\eta_{P_S(X)}^{-1}(\Box (\Box U))\in \mathcal O(P_S(P_S(X)))$. Thus  $\bigcup : P_S(P_S(X)) \longrightarrow P_S(X)$ is continuous.
\end{proof}

As in \cite{E_20182}, a topological space $X$ is \emph{locally hypercompact} if for each $x\in X$ and each open neighborhood $U$ of $x$, there is  $\ua F\in \mathbf{Fin}~X$ such that $x\in\ii\,\ua F\subseteq\ua F\subseteq U$. A space $X$ is called a $C$-\emph{space} if for each $x\in X$ and each open neighborhood $U$ of $x$, there is $u\in X$ such that $x\in\ii\,\ua u\subseteq\ua u\subseteq U$). A set $K\subseteq X$ is called \emph{supercompact} if for
any arbitrary family $\{U_i : i\in I\}\subseteq \mathcal O(X)$, $K\subseteq \bigcup_{i\in I} U_i$  implies $K\subseteq U$ for some $i\in I$. It is easy to check that the supercompact saturated sets of $X$ are exactly the sets $\ua x$ with $x \in X$ (see \cite[Fact 2.2]{Klause-Heckmann}). It is well-known that $X$ is a $C$-space if{}f $\mathcal O(X)$ is a \emph{completely distributive} lattice (cf. \cite{E_2009}). A space $X$ is called \emph{core compact} if $\mathcal O(X)$ is a \emph{continuous lattice} (cf. \cite{redbook}).

\begin{theorem}\label{SoberLC=CoreC}\emph{(\cite{redbook})} Let $X$ be a sober space. Then $X$ is locally compact if{}f $X$ is core compact.
\end{theorem}

For a $T_0$ space $X$ and a nonempty subset $C$ of $X$, it is easy to see that $C$ is compact if{}f $\ua C\in \mk (X)$. The following result is well-known (see, e.g., \cite[pp.2068]{E_2009}) .

\begin{lemma}\label{COMPminimalset} Let $X$ be a $T_0$ space and $C\in \mk (X)$. Then $C=\ua min(C)$ and  $min(C)$ is compact.
\end{lemma}

For a $T_0$ space $X$, $\mathcal{U}\subseteq \mathcal O(X)$ is called an \emph{open filter} if $\mathcal U\in \sigma (\mathcal O(X))\bigcap \mathrm{Filt}(\mathcal O(X))$. For $K\in \mk (X)$, let $\Phi (K)=\{U\in \mathcal O(X) : K\subseteq U\}$. Then $\Phi (K)\in \mathrm{OFilt(\mathcal O(X))}$ and $K=\bigcap \Phi (K)$. Obviously, $\Phi : \mk (X) \longrightarrow \mathrm{OFilt(\mathcal O(X))}, K\mapsto \Phi (K)$, is an order embedding. It is well-known that $\Phi$ is an order isomorphism if{}f $X$ is sober (see \cite{redbook, Hofmann-Mislove} or Theorem \ref{Hofmann-Mislove theorem} in this paper).

\section{$d$-spaces and directed closure spaces}

In this section, we give some equational characterizations of $d$-spaces. Based on directed sets, we introduce the concept of directed closure spaces, and discuss some basic properties of them.

A $T_0$ space $X$ is called a \emph{d-space} (or \emph{monotone convergence space}) if $X$ (with the specialization order) is a dcpo
 and $\mathcal O(X) \subseteq \sigma(X)$ (cf. \cite{redbook, Wyler}).

\begin{definition}\label{directbound} A $T_0$ space $X$ is called \emph{directed bounded}, $d$-\emph{bounded} for short, if for any $D\in \mathcal D(X)$, $D$ has an upper bound in $X$, that is, there is an $x\in X$ such that $D\subseteq \da x=\overline {\{x\}}$.
\end{definition}

Clearly, we have the following implications:

\begin{center}
$d$-space $\Rightarrow$ direct completeness $\Rightarrow$ $d$-boundedness.
\end{center}

For a poset $P$ with a largest element $\top$, any \emph{order compatible} topology $\tau$ on $P$ (that is, $\leq_{\tau}$ agrees with the original order on $P$) is $d$-bounded.

\begin{proposition}\label{d-bounded} For a $T_0$ space $X$, the following conditions are equivalent:
\begin{enumerate}[\rm (1)]
	\item $X$ is $d$-bounded.
    \item For any $D\in \mathcal D(X)$, $D^{\ua}=\bigcap\limits_{d\in D}\ua d\neq\emptyset$.
    \item  For any $D\in \mathcal D(X)$ and $A\in \mathcal C(X)$, if $D\subseteq A$, then $\bigcap\limits_{d\in D}\ua (A\cap\ua d)\neq\emptyset$.
    \item For any $D\in \mathcal D(X)$ and $A\in \ir_c(X)$, if $D\subseteq A$, then $\bigcap\limits_{d\in D}\ua (A\cap\ua d)\neq\emptyset$.
\end{enumerate}
\end{proposition}
\begin{proof} (1) $\Leftrightarrow$ (2) and (3) $\Rightarrow$ (4): Trivial.

(2) $\Rightarrow$ (3): $\emptyset\neq\bigcap\limits_{d\in D}\ua d\subseteq\bigcap\limits_{d\in D}\ua (A\cap\ua d)$.

(4) $\Rightarrow$ (2): Since $D\in \mathcal D(X)$, $\overline{D}\in \ir_c(X)$. By condition (4), $\bigcap\limits_{d\in D}\ua d=\bigcap\limits_{d\in D}\ua (\overline{D}\cap\ua d)\neq\emptyset$.
\end{proof}

\begin{proposition}\label{d-spacecharac1} For a $T_0$ space $X$, the following conditions are equivalent:
\begin{enumerate}[\rm (1) ]
	        \item $X$ is a $d$-space.
            \item $\mathcal D_c(X)=\mathcal S_c(X)$.
            \item  For any $D\in \mathcal D(X)$ and $U\in \mathcal O(X)$, $\bigcap\limits\limits_{d\in D}\ua d\subseteq U$ implies $\ua d \subseteq U$ \emph{(}i.e., $d\in U$\emph{)} for some $d\in D$.
            \item  For any filtered family $\mathcal K\subseteq \mathcal S^u(X)$ and $U\in \mathcal O(X)$, $\bigcap \mathcal K\subseteq U$ implies $K\subseteq U$ for some $K\in\mathcal K$.
            \item  For any $D\in \mathcal D(X)$ and $A\in \mathcal C(X)$, if $D\subseteq A$, then $A\cap\bigcap\limits_{d\in D}\ua d\neq\emptyset$.
            \item  For any $D\in \mathcal D(X)$ and $A\in \ir_c(X)$, if $D\subseteq A$, then $A\cap\bigcap\limits_{d\in D}\ua d\neq\emptyset$.
            \item  For any $D\in \mathcal D(X)$, $\overline{D}\cap\bigcap\limits_{d\in D}\ua d\neq\emptyset$.
\end{enumerate}
\end{proposition}
\begin{proof} (1) $\Leftrightarrow$ (2): Clearly, (1) $\Rightarrow$ (2). Conversely, if condition (2) holds, then for each $D\in \mathcal D(X)$ and $A\in \mathcal C(X)$ with $D\subseteq A$, there is $x\in X$ such that $\overline{D}=\overline{\{x\}}$, and consequently, $\bigvee D=x$ and $\bigvee D \in A$ since $\overline{D}\subseteq A$. Thus $X$ is a dcpo and $\mathcal O(X)\subseteq \sigma(X)$, proving $X$ is a $d$-space.

(1) $\Rightarrow$ (3): Since $X$ is a $d$-space, $\ua \bigvee D=\bigcap\limits_{d\in D}\ua d\subseteq U\in \sigma(X)$. Therefore, $\bigvee D\in U$, and whence
$d\in U$ for some $d\in D$.

(3) $\Leftrightarrow$ (4): For $\mathcal K=\{\ua x_i : i\in I\}\subseteq \mathcal S^u(X)$, $\mathcal K$
is filtered in $\mathcal S^u(X)$ with the Smyth order if{}f $\{x_i : i\in I\}\in \mathcal D(X)$.

(3) $\Rightarrow$ (5): If $A\cap\bigcap\limits_{d\in D}\ua d=\emptyset$, then $\bigcap\limits_{d\in D}\ua d \subseteq X\setminus A$. By condition (3), $\ua d\subseteq X\setminus A$ for some $d\in D$, which is in contradiction with $D\subseteq A$.

(5) $\Rightarrow$ (6) $\Rightarrow$ (7): Trivial.

(7) $\Rightarrow$ (1): For each $D\in \mathcal D(X)$ and $A\in \mathcal C(X)$ with $D\subseteq A$, by condition (7), $\overline{D}\cap\bigcap\limits_{d\in D}\ua d\neq\emptyset$. Select an $x\in \overline{D}\cap\bigcap\limits_{d\in D}\ua d$. Then $D\subseteq \da x\subseteq \overline{D}$, and hence $\overline{D}=\da x$ and $\bigvee D=x$. Therefore, $\bigvee D\in A$ because $\overline{\{\bigvee D\}}=\overline{D}\subseteq A$. Thus $X$ is a $d$-space.
\end{proof}

In the following, we shall give some equational characterizations of $d$-spaces.

\begin{proposition}\label{d-spacecharac2} For a $T_0$ space $X$, the following conditions are equivalent:
\begin{enumerate}[\rm (1) ]
	        \item $X$ is a $d$-space.
            \item  $X$ is $d$-bounded \emph{(}especially, $X$ is a dcpo\emph{)}, and $\ua \left(A\cap\bigcap\limits_{d\in D} \ua d\right)=\bigcap\limits_{d\in D}\ua (A\cap \ua d)$ for any $D\in \mathcal D(X)$ and $A\in \mathcal C(X)$.
            \item  $X$ is $d$-bounded \emph{(}especially, $X$ is a dcpo\emph{)}, and $\ua \left(A\cap\bigcap \mathcal K\right)=\bigcap\limits_{K\in\mathcal K}\ua (A\cap K)$ for any filtered family $\mathcal K\subseteq \mathcal S^u(X)$ and $A\in \mathcal C(X)$.
            \item  $X$ is $d$-bounded \emph{(}especially, $X$ is a dcpo\emph{)}, and $\ua \left(A\cap\bigcap\limits_{d\in D} \ua d\right)=\bigcap\limits_{d\in D}\ua (A\cap \ua d)$ for any $D\in \mathcal D(X)$ and $A\in \ir_c(X)$.
            \item  $X$ is $d$-bounded \emph{(}especially, $X$ is a dcpo\emph{)}, and $\ua \left(A\cap\bigcap \mathcal K\right)=\bigcap\limits_{K\in\mathcal K}\ua (A\cap K)$ for any filtered family $\mathcal K\subseteq \mathcal S^u(X)$ and $A\in \ir_c(X)$.
\end{enumerate}
\end{proposition}
\begin{proof}  (1) $\Rightarrow$ (2): Since $X$ is a $d$-space, $X$ is a dcpo and $\mathcal O(X)\subseteq\sigma (X)$. For $D\in \mathcal D(X)$ and $A\in \mathcal C(X)$, clearly, $\ua \left(A\cap\bigcap\limits_{d\in D} \ua d\right)\subseteq\bigcap\limits_{d\in D}\ua (A\cap \ua d)$. Conversely, if $x\not\in \ua \left(A\cap\bigcap\limits_{d\in D}\ua d\right)$, that is, $\da x\cap A\cap\bigcap\limits_{d\in D}\ua d=\emptyset$, then $\ua \bigvee D=\bigcap\limits_{d\in D}\ua d\subseteq X\setminus \da x\cap A\in \sigma (X)$, and whence $d\in X\setminus \da x\cap A$ for some $d\in D$, i.e., $x\not\in \ua (A\cap\ua d)$. Therefore, $x\not\in \bigcap\limits_{d\in D}\ua (A\cap \ua d)$. Thus $\ua \left(A\cap\bigcap\limits_{d\in D} \ua d\right)=\bigcap\limits_{d\in D}\ua (A\cap \ua d)$.

(2) $\Leftrightarrow$ (3) and (4) $\Leftrightarrow$ (5): For $\mathcal K=\{\ua x_i : i\in I\}\subseteq \mathcal S^u(X)$, $\mathcal K$
is filtered in $\mathcal S^u(X)$ with the Smyth order if{}f $\{x_i : i\in I\}\in \mathcal D(X)$.

(2) $\Rightarrow$ (4): Trivial.

(4) $\Rightarrow$ (1): For each $D\in \mathcal D(X)$, by condition (4), $\emptyset \neq D^{\ua}=\bigcap\limits_{d\in D}\ua d=\bigcap\limits_{d\in D}\ua (\overline{D}\cap \ua d)=\ua \left(\overline{D}\cap\bigcap\limits_{d\in D} \ua d\right)$. By Proposition \ref{d-spacecharac1}, $X$ is a $d$-space.
\end{proof}

\begin{theorem}\label{d-spacemapcharac} Let $X$ be a $T_0$ space and $\mathbf{K}$ a full subcategory of $\mathbf{Top}_0$ containing $\mathbf{Sob}$. Then the following conditions are equivalent:
\begin{enumerate}[\rm (1)]
		\item $X$ is a $d$-space.
        \item For every continuous mapping $f:X\longrightarrow Y$ from $X$ to a $T_0$ space $Y$ and any $D\in \mathcal D(X)$, $\ua f\left(\bigcap\limits_{d\in D}\ua d\right)=\bigcap\limits_{d\in D}\ua f(\ua d)=\bigcap\limits_{d\in D} \ua f(d)$.
        \item For every continuous mapping $f:X\longrightarrow Y$ from $X$ to a $\mathbf{K}$-space $Y$ any $D\in \mathcal D(X)$, $\ua f\left(\bigcap\limits_{d\in D}\ua d\right)=\bigcap\limits_{d\in D}\ua f(\ua d)=\bigcap\limits_{d\in D} \ua f(d)$.
        \item For every continuous mapping $f:X\longrightarrow Y$ from $X$ to a sober space $Y$ and any $D\in \mathcal D(X)$, $\ua f\left(\bigcap\limits_{d\in D}\ua d\right)=\bigcap\limits_{d\in D}\ua f(\ua d)=\bigcap\limits_{d\in D} \ua f(d)$.
\end{enumerate}
\end{theorem}
\begin{proof} (1) $\Rightarrow$ (2): Since $X$ is a $d$-space and $f$ is order-preserving, we have $\ua f\left(\bigcap\limits_{d\in D}\ua d\right)=\ua f(\ua \bigvee D)=\ua f(\bigvee D)$. Obviously, $\ua f(\bigvee D)\subseteq \bigcap\limits_{d\in D} \ua f(d)$. On the other hand, if $y\in \bigcap\limits_{d\in D} \ua f(d)$, then $d\in f^{-1}(\da y)=f^{-1}(\overline{\{y\}})\in \mathcal C(X)\subseteq \mathcal C(\Sigma X)$ for all $d\in D$, and hence $\bigvee D\in f^{-1}(\da y)$, that is, $y\in \ua f(\bigvee D)$. Thus $f(\bigvee D)=\bigvee f(D)$, and whence $\ua f\left(\bigcap\limits_{d\in D}\ua d\right)=\ua f(\bigvee D)=\ua \bigvee f(D)=\bigcap\limits_{d\in D}\ua f(\ua d)=\bigcap\limits_{d\in D} \ua f(d)$.

(2) $\Rightarrow$ (3)$\Rightarrow$ (4): Trivial.

(4) $\Rightarrow$ (1):  Let $\eta_X : X \rightarrow X^s$ ($=P_H(\ir_c(X))$) be the canonical topological embedding from $X$ into its soberification. For $D\in \mathcal D(X)$, by condition (4) we have $\overline{D}\in \bigcap\limits_{d\in D}\ua_{\ir_c(X)} \eta_X (d)=\bigcap\limits_{d\in D}\ua_{\ir_c(X)} \eta_X (\ua d)=\ua_{\ir_c(X)} \eta_X\left(\bigcap\limits_{d\in D}\ua d\right)=\ua_{\ir_c(X)} \eta_X (D^{\ua})$, and whence there is a $c\in D^{\ua}$ such that $\overline{\{c\}}\subseteq \overline{D}$. Therefore, $\overline{D}=\overline{\{c\}}$. By Proposition \ref{d-spacecharac1}, $X$ is a $d$-space.
\end{proof}

\begin{corollary}\label{d-spacemapcharac2} Let $X$ be a $T_0$ space and $\mathbf{K}$ a full subcategory of $\mathbf{Top}_0$ containing $\mathbf{Sob}$. Then  the following conditions are equivalent:
\begin{enumerate}[\rm (1)]
		\item $X$ is a $d$-space.
        \item $X$ is a dcpo, and for every continuous mapping $f:X\longrightarrow Y$ from $X$ to a $T_0$ space $Y$ and any $D\in \mathcal D(X)$, $ f(\bigvee D)=\bigvee f(D)$.
        \item $X$ is a dcpo, and for every continuous mapping $f:X\longrightarrow Y$ from $X$ to a $\mathbf{K}$-space $Y$ and any $D\in \mathcal D(X)$, $ f(\bigvee D)=\bigvee f(D)$.
        \item $X$ is a dcpo, and for every continuous mapping $f:X\longrightarrow Y$ from $X$ to a sober space $Y$ and any $D\in \mathcal D(X)$, $ f(\bigvee D)=\bigvee f(D)$.
\end{enumerate}
\end{corollary}
\begin{proof} (1) $\Rightarrow$ (2): Since $X$ is a $d$-space, $X$ is a dcpo. By the proof of (1) $\Rightarrow$ (2) in Theorem \ref{d-spacemapcharac}, $f(\bigvee D)=\bigvee f(D)$.

(2) $\Rightarrow$ (3)$\Rightarrow$ (4): Trivial.

(4) $\Rightarrow$ (1): For every continuous mapping $f:X\longrightarrow Y$ from $X$ to a sober space $Y$ and any $D\in \mathcal D(X)$, by condition (4) we have $\ua f\left(\bigcap\limits_{d\in D}\ua d\right)=\ua f(\bigvee D)=\ua \bigvee f(D)=\bigcap\limits_{d\in D} \ua f(d)=\bigcap\limits_{d\in D}\ua f(\ua d)$, and whence by Theorem \ref{d-spacemapcharac}, $X$ is a $d$-space.
\end{proof}

\begin{definition}\label{DCspace}
	A $T_0$ space $X$ is called a \emph{directed closure space}, $\mathsf{DC}$ \emph{space} for short, if $\ir_c(X)=\mathcal{D}_c(X)$, that is, for each $A\in \ir_c(X)$, there exists a directed subset of $X$ such that $A=\overline{D}$.
\end{definition}

The following result follows directly from the definition of $\mathsf{DC}$ spaces.

\begin{proposition}\label{DCclosed}
	A closed subspace of a $\dc$ space is a $\dc$ space.
\end{proposition}

\begin{lemma}\label{DCimag}
	If $f : X \longrightarrow Y$ is continuous and $A\in\mathcal D_c (X)$, then $\overline{f(A)}\in \mathcal D_c (Y)$.
\end{lemma}
\begin{proof} Since $A\in\mathcal D_c (X)$, there is a $D\in \mathcal D(X)$ such that $A=\overline{D}$, and whence $f(D)\in \mathcal D(Y)$ and $\overline{f(A)}=\overline{f(\overline{D})}=\overline{f(D)}\in \mathcal D_c (Y)$.
\end{proof}

\begin{proposition}\label{DCretract}
	A retract of a $\dc$ space is a $\dc$ space.
\end{proposition}
\begin{proof}
	Assume $X$ is a $\dc$ space and $Y$ a retract of $X$.  Then there are continuous mappings $f:X\longrightarrow Y$ and $g:Y\longrightarrow X$ such that $f\circ g={\rm id}_Y$.  Let $B\in \ir_c(Y)$. Then $\overline{g(B)}\in\ir_c(X)$ by Lemma \ref{irrimage}. Since $X$ is  a $\dc$ space, $\overline{g(B)}\in D_c(X)$. By Lemma \ref{DCimag}, $B=\overline{fg(B)}=\overline{f\left(\overline{g(B)}\right)}\in\mathcal D_c (Y)$. Therefore, $Y$ is a $\dc$ space.
\end{proof}

By Corollary \ref{irrcprod} and Lemma \ref{DCimag}, we get the following result.

\begin{proposition}\label{DCsetprod}
	Let	$\{X_i:i\in I\}$ be a family of $T_0$ spaces and $X=\prod_{i\in I}X_i$. If $A\in \mathcal D_c(X)$, then $A=\prod_{i\in I}p_i(A)$, and $p_i(A)\in \mathcal D_c(X_i)$ for each $i\in I$.
\end{proposition}

\begin{proposition}\label{DCprodset}
	Let	$X=\prod_{i\in I}X_i$ be the product of a family $\{X_i:i\in I\}$ of $T_0$ spaces and $A_i\subseteq X_i$ for each $i\in I$. Then the following two conditions are equivalent:
\begin{enumerate}[\rm (1)]
	\item $\prod_{i\in I}A_i\in \mathcal D_c(X)$.
	\item $A_i\in \mathcal D_c (X_i)$ for each $i\in I$.
\end{enumerate}
\end{proposition}
\begin{proof}  (1) $\Rightarrow$ (2): By Proposition \ref{DCsetprod}.

(2) $\Rightarrow$ (1): For each $i\in I$, by $A_i\in \mathcal D_c (X_i)$, there is a $D_i\in \mathcal D(X_i)$ such that $A_i=\cl_{X_i}{D_i}$. Let $D=\prod_{i\in I}D_i$. Then $D\in \mathcal D(X)$. By \cite[Proposition 2.3.3]{Engelking}, $\prod_{i\in I}A_i=\prod_{i\in I}\cl_{X_i}D_i=\cl_X\prod_{i\in I}D_i=\cl_X D\in \mathcal D_c(X)$.
\end{proof}

\begin{corollary}\label{DCprod}
	Let $\{X_i:i\in I\}$ be a family of $T_0$ spaces. Then the following two conditions are equivalent:
	\begin{enumerate}[\rm(1)]
		\item The product space $\prod_{i\in I}X_i$ is a $\mathsf{DC}$ space.
		\item For each $i \in I$, $X_i$ is a $\mathsf{DC}$ space.
	\end{enumerate}
\end{corollary}
\begin{proof}	
	(1) $\Rightarrow$ (2):  For each $i \in I$, $X_i$ is a retract of $\prod_{i\in I}X_i$. By Proposition \ref{DCretract}, $X_i$ is a $\mathsf{DC}$ space.
	
	(2) $\Rightarrow$ (1): Let $X=\prod_{i\in I}X_i$. Suppose $A\in \ir_c (X)$. Then for each $i \in I$, by Corollary \ref{irrcprod}, $A=\prod_{i\in I}p_i (A)$ and $p_i(A)\in \ir_c(X_i)$, and whence $p_i(A)\in \mathcal D_c(X_i)$ because $X_i$ is a $\dc$ space. By Proposition \ref{DCprodset}, $A=\prod_{i\in I}p_i (A)\in \mathcal D_c(X)$. Thus $X$ is a $\dc$ space.
\end{proof}

\section{Rudin's Lemma and Rudin spaces}

Rudin's Lemma is a useful tool in topology and plays a crucial role in domain theory (see [3-14, 24, 25, 29]). Rudin \cite{Rudin} proved her lemma by transfinite methods, using the Axiom of Choice.
Heckman and Keimel \cite{Klause-Heckmann} presented the following topological variant of Rudin's Lemma.

\begin{lemma}\label{t Rudin} \emph{(Topological Rudin's Lemma)} Let $X$ be a topological space and $\mathcal{A}$ an
irreducible subset of the Smyth power space $P_S(X)$. Then every closed set $C {\subseteq} X$  that
meets all members of $\mathcal{A}$ contains an minimal irreducible closed subset $A$ that still meets all
members of $\mathcal{A}$.
\end{lemma}

\begin{corollary}\label{rudinmeet}
Let $X$ be a $T_0$ space. If $\mathcal A\in\ir_c(P_S(X))$, then there exists a family $\{A_i:i\in I\}$ of minimal irreducible cosed sets such that $\mathcal A=\bigcap_{i\in I}\Diamond A_i$.
\end{corollary}

Applying Lemma \ref{t Rudin} to the Alexandroff topology on a poset $P$, one obtains  the original Rudin's Lemma.

\begin{corollary}\label{rudin} \emph{(Rudin's Lemma)} Let $P$ be a poset, $C$ a nonempty lower subset of $P$ and $\mathcal F\in \mathbf{Fin}~P$ a filtered family with $\mathcal F\subseteq\Diamond C$. Then there exists a directed subset $D$ of $C$ such that $\mathcal F\subseteq \Diamond\da D$.
\end{corollary}

For a $T_0$ space $X$ and $\mathcal{K}\subseteq \mathord{\mathsf{K}}(X)$, let $M(\mathcal{K})=\{A\in \mathcal C(X) : K\bigcap A\neq\emptyset \mbox{~for all~} K\in \mathcal{K}\}$ (that is, $\mathcal A\subseteq \Diamond A$) and $m(\mathcal{K})=\{A\in \mathcal C(X) : A \mbox{~is a minimal menber of~} M(\mathcal{K})\}$.

By the proof of \cite[Lemma 3.1]{Klause-Heckmann}, we have the following result.

\begin{lemma}\label{t ruding}  Let $X$ be a $T_0$ space and $\mathcal{K}\subseteq \mathord{\mathsf{K}}(X)$. If $C\in M(\mathcal{K})$, then there is a closed subset $A$ of $C$ such that $C\in m(\mathcal{K})$.
\end{lemma}

The following result shows that the reverse of Lemma \ref{t Rudin} holds.

\begin{lemma}\label{rudinequiv}
	Let $X$ be a $T_0$ space and $\mathcal A$ a nonempty subset of $P_S(X)$. Then the following conditions are equivalent:
	\begin{enumerate}[\rm (1)]
		\item $\mathcal A$ is irreducible;
		\item $\forall A\in\mathcal C(X)$, if $\mathcal A\subseteq \Diamond A$, then there exists a minimal irreducible closed set $C\subseteq A$ such that $\mathcal A\subseteq\Diamond C$.
	\end{enumerate}
\end{lemma}
\begin{proof} (1) $\Rightarrow$ (2): By Lemma \ref{t Rudin}.

(2) $\Rightarrow$ (1): Let $\mathcal A\subseteq \mathcal A_1\cup\mathcal A_2$ with $\{\mathcal A_1,\mathcal A_2\}\subseteq\mathcal C(P_S(X))$. Then there exists $\{A_i:i\in I\}\subseteq\mathcal C(X)$ and $\{B_j:j\in J\}\subseteq\mathcal C(X)$ such that $\mathcal A=\bigcap_{i\in I}\Diamond A_i$ and $\bigcap_{j\in J}\Diamond B_j$. Suppose, on the contrary, $\mathcal A\nsubseteq \mathcal A_1$ and $\mathcal A\nsubseteq \mathcal A_2$. Then there exists $(i,j)\in I\times J$ such that $\mathcal A\nsubseteq\Diamond\mathcal A_{i}$ and $\mathcal A\nsubseteq\Diamond B_{j}$. Note that $\mathcal A_1\subseteq \Diamond A_{i}$ and $\mathcal A_2\subseteq \Diamond B_{j}$, so $\mathcal A\subseteq \Diamond A_{i}\cup\Diamond B_{j}=\Diamond(A_{i}\cup B_{j})$. By (2), there exists a minimal irreducible closed set $C\subseteq A_{i}\cup B_{j}$ such that $\mathcal A\subseteq\Diamond C$. This implies that $C\subseteq A_{i}$ or $C\subseteq B_{j}$, so $\mathcal A\subseteq \Diamond C\subseteq \Diamond A_{i}$ or $\mathcal A\subseteq \Diamond C\subseteq \Diamond B_{j}$, a contradiction. Therefore $\mathcal A\subseteq\mathcal A_1$ or $\mathcal A\subseteq\mathcal A_2$. Thus $\mathcal A$ is irreducible.
\end{proof}

In the following, based on topological Rudin's Lemma, we introduce and investigate a new kind of spaces - Rudin spaces, which lie between $\dc$ spaces and sober spaces. It is proved that closed subspaces, retracts and products of Rudin spaces are again Rudin spaces.

\begin{definition}\label{rudinset} (\cite{Shenchon})
		Let $X$ be a $T_0$ space. A nonempty subset  $A$  of $X$  is said to have the \emph{Rudin property}, if there exists a filtered family $\mathcal K\subseteq \mathord{\mathsf{K}}(X)$ such that $\overline{A}\in m(\mathcal K)$ (that is,  $\overline{A}$ is a minimal closed set that intersects all members of $\mathcal K$). Let $\mathsf{RD}(X)=\{A\in \mathcal C(X) : A\mbox{~has Rudin property}\}$. The sets in $\mathsf{RD}(X)$ will also be called \emph{Rudin sets}.
\end{definition}

The Rudin property is called the \emph{compactly filtered property} in \cite{Shenchon}. In order to emphasize its origin from (topological) Rudin's Lemma, here we call such a property the Rudin property. Clearly, $A$ has Rudin property if{}f $\overline{A}$ has Rudin property (that is, $\overline{A}$ is a Rudin set).

\begin{definition}\label{rudinspace}  A $T_0$ space $X$ is called a \emph{Rudin space}, $\mathsf{RD}$ \emph{space} for short, if $\ir_c(X)=\mathsf{RD}(X)$, that is, every irreducible closed set of $X$ is a Rudin set. The category of all Rudin spaces with continuous mappings is denoted by $\mathbf{Top}_{r}$.
\end{definition}

\begin{lemma}\label{DRIsetrelation}\emph{(\cite{Shenchon})}
	Let $X$ be a $T_0$ space. Then $\mathcal D_c(X)\subseteq \mathsf{RD}(X)\subseteq \ir_c(X)$.
\end{lemma}
\begin{proof} By Lemma \ref{t Rudin} we have $\mathsf{RD}(X)\subseteq \ir_c(X)$. Now we prove that the closure of a directed subset $D$ of $X$ is a Rudin set. Let
$\mathcal K_D=\{\ua d : d\in D\}$. Then $\mathcal K_D\subseteq \mathord{\mathsf{K}}(X)$ is filtered and $\overline{D}\in M(\mathcal K_D)$. If $A\in M(\mathcal K_D)$, then $d\in A$ for every $d\in D$, and hence $\overline{D}\subseteq A$. So $\overline{D}\in m(\mathcal K_D)$. Therefore $\overline{D}\in \mathsf{RD}(X)$.
\end{proof}

\begin{proposition}\label{rudinclosed} A closed subspace of a Rudin space is a Rudin space.
\end{proposition}
\begin{proof} Let $X$ be a Rudin space and $A\in \mathcal C(X)$. For $B\in \ir_c(A)$, we have $B\in \ir_c(X)$ by Lemma \ref{irrsubspace}. Since $X$ be a Rudin space, there exists a filtered family $\mathcal K\subseteq \mathord{\mathsf{K}}(X)$ such that $B\in m(\mathcal K)$. Let $\mathcal{K}_B=\{\ua_A (K\cap B) : K\in \mathcal K\}$. Then $\mathcal{K}_B\subseteq \mathord{\mathsf{K}}(A)$ is filtered. For each $K\in \mathcal K$, since $K\cap B\neq\emptyset$, we have $\emptyset\neq K\cap B\subseteq \ua_A (K\cap B)\cap B$. So $B\in M(\mathcal{K}_B)$. If $C$ is a closed subset of $B$ with $C\in M(\mathcal{K}_B)$, then $C\cap\ua_A (K\cap B)\neq\emptyset$ for every $K\in \mathcal K$. So $K\cap B\cap C=K\cap C\neq\emptyset$ for all $K\in \mathcal K$. It follows that $B=C$ by the minimality of $B$, and consequently, $B\in m(\mathcal K_B)$. Whence $A$ is a Rudin space.
\end{proof}

\begin{lemma}\label{rudinimage}\emph{(\cite{Shenchon})}
	Let $X, Y$ be two $T_0$ spaces and $f : X\longrightarrow Y$ a continuous mapping. If $A\in \mathsf{RD}(X)$, then $\overline{f(A)}\in \mathsf{RD}(Y)$.
\end{lemma}
\begin{proof} It has been proved in \cite{Shenchon}. Here we give a more direct proof. Since $A\in \mathsf{RD}(X)$, there exists a filtered family $\mathcal K\subseteq \mathord{\mathsf{K}}(X)$ such that $A\in m(\mathcal K)$. Let $\mathcal{K}_f=\{\ua f(K\cap A) : K\in \mathcal K\}$. Then $\mathcal{K}_f\subseteq \mathord{\mathsf{K}}(Y)$ is filtered. For each $K\in \mathcal K$, since $K\cap A\neq\emptyset$, we have $\emptyset\neq f(K\cap A)\subseteq \ua f(K\cap A)\cap \overline{f(A)}$. So $\overline{f(A)}\in M(\mathcal{K}_f)$. If $B$ is a closed subset of $\overline{f(A)}$ with $B\in M(\mathcal{K}_f)$, then $B\cap\ua f(K\cap A)\neq\emptyset$ for every $K\in \mathcal K$. So $K\cap A\cap f^{-1}(B)\neq\emptyset$ for all $K\in \mathcal K$. It follows that $A=A\cap f^{-1}(B)$ by the minimality of $A$, and consequently, $\overline{f(A)}\subseteq B$. Therefore, $\overline{f(A)}=B$. Thus $\overline{f(A)}\in \mathsf{RD}(Y)$.
\end{proof}

\begin{corollary}\label{rudinretract} A retract of a Rudin space is a Rudin space.
\end{corollary}
\begin{proof} Suppose that $Y$ is a retract of a Rudin space $X$. Then there are continuous mappings $f : X\longrightarrow Y$ and $g : Y\longrightarrow X$ such that $f\circ g=id_Y$. Let $B\in \ir_c(Y)$, then by Lemma \ref{irrimage}, $\overline{f(\overline{g(B)})}\in \ir_c(Y)$. Since $X$ is a Rudin space, $\overline{g(B)}\in \mathsf{RD}(X)$. By Lemma \ref{rudinimage}, $B=\overline{f\circ g(B)}=\overline{f(\overline{g(B)})}\in \mathsf{RD}(Y)$. Thus $Y$ is a Rudin space.
\end{proof}

\begin{proposition}\label{rudinwf}
	Let $X$ be a $T_0$ space and  $Y$ a well-filtered space. If $f : X\longrightarrow Y$ is continuous and $A\subseteq X$ has Rudin property, then there exists a unique $y_A\in X$ such that $\overline{f(A)}=\overline{\{y_A\}}$.
\end{proposition}
\begin{proof}
Since $A$ has Rudin property, there exists a filtered family $\mathcal K\subseteq \mathord{\mathsf{K}}(X)$ such that $\overline{A}\in m(\mathcal K)$. Let $\mathcal{K}_f=\{\ua f(K\cap \overline{A}) : K\in \mathcal K\}$. Then $\mathcal{F}_f\subseteq \mathord{\mathsf{K}}(Y)$ is filtered. By the proof of Lemma \ref{rudinimage}, $\overline{f(A)}\in m(\mathcal{K}_f)$. Since $Y$ is well-filtered, we have $\bigcap_{K\in \mathcal{K}}\ua f(K\cap \overline{A})\cap \overline{f(A)}\neq\emptyset$. Select a $y_A\in \bigcap_{K\in \mathcal K} \ua f(K\cap \overline{A})\cap \overline{f(A)}$. Then $\overline{\{y_A\}}\subseteq \overline{f(A)}$ and $K\cap \overline{A}\cap f^{-1}(\overline{\{y_A\}})\neq\emptyset$ for all $K\in \mathcal K$. It follows that $\overline{A}=\overline{A}\cap f^{-1}(\overline{\{y_A\}})$ by the minimality of $\overline{A}$, and consequently, $\overline{f(A)}\subseteq \overline{\{y_A\}}$. Therefore, $\overline{f(A)}=\overline{\{y_A\}}$. The uniqueness of $y_A$ follows from the $T_0$ separation of $Y$.
\end{proof}

\begin{lemma}\label{Rudinsetprod}\emph{(\cite{Shenchon})}
	Let	$X=\prod_{i\in I}X_i$ be the product of a family $\{X_i:i\in I\}$ of $T_0$ spaces and $A\subseteq X$. Then the following conditions are equivalent:
\begin{enumerate}[\rm (1)]
	\item $A$ is a Rudin set.
	\item $p_i(A)$ is a Rudin set for each $i\in I$.
\end{enumerate}
\end{lemma}

\begin{theorem}\label{rudinprod}
	Let $\{X_i:i\in I\}$ be a family of $T_0$ spaces. Then the following two conditions are equivalent:
	\begin{enumerate}[\rm(1)]
		\item The product space $\prod_{i\in I}X_i$ is a Rudin space.
		\item For each $i \in I$, $X_i$ is a Rudin space.
	\end{enumerate}
\end{theorem}
\begin{proof}	
	(1) $\Rightarrow$ (2):  For each $i \in I$, $X_i$ is a retract of $\prod_{i\in I}X_i$. By Corollary \ref{rudinretract}, $X_i$ is a Rudin space.
	
	(2) $\Rightarrow$ (1): Suppose $A\in \ir_c(\prod_{i\in I}X_i)$. Then for each $i \in I$, since $X_i$ is a Rudin space, $p_i(A)\in \mathsf{RD}(X_i)$ by Corollary \ref{irrcprod}, and consequently, by Corollary \ref{irrcprod} and Lemma \ref{Rudinsetprod}, $A=\prod_{i\in I}p_i(A)\in \mathsf{RD}(\prod_{i\in I}X_i)$. Therefore,	$\prod_{i\in I}X_i$ is a Rudin space.
\end{proof}

\section{Well-filtered spaces and sober spaces}

In this section, we formulate and prove some equational characterizations of well-filtered spaces and sober spaces.

\begin{theorem}\label{WFequat1} Let $X$ be a $T_0$ space and $\mathbf{K}$ a full subcategory of $\mathbf{Top}_0$ containing $\mathbf{Sob}$. Then the following conditions are equivalent:
\begin{enumerate}[\rm (1)]
		\item $X$ is well-filtered.
        \item For every continuous mapping $f:X\longrightarrow Y$ from $X$ to a $T_0$ space $Y$ and a filtered family $\mathcal K\subseteq \mk(X)$, $\ua f\left(\bigcap\mathcal K\right)=\bigcap\limits_{K\in\mathcal K}\ua f(K)$.
        \item For every continuous mapping $f:X\longrightarrow Y$ from $X$ to a $\mathbf{K}$-space $Y$ and a filtered family $\mathcal K\subseteq \mk(X)$, $\ua f\left(\bigcap\mathcal K\right)=\bigcap\limits_{K\in\mathcal K}\ua f(K)$.
        \item For every continuous mapping $f:X\longrightarrow Y$ from $X$ to a sober space $Y$ and a filtered family $\mathcal K\subseteq \mk(X)$, $\ua f\left(\bigcap\mathcal K\right)=\bigcap\limits_{K\in\mathcal K}\ua f(K)$.
\end{enumerate}
\end{theorem}
\begin{proof} (1) $\Rightarrow$ (2): It is proved in \cite{Esc-Laws-Simp2004} for sober spaces and the proof is valid for well-filtered spaces (see \cite[Lemma 8.1]{Esc-Laws-Simp2004}). For the sake of completeness, we present the proof here. It needs only to check $\bigcap\limits_{K\in\mathcal K}\ua f(K)\subseteq \ua f\left(\bigcap\mathcal K\right)$. Let $y\in\bigcap\limits_{K\in\mathcal K}\ua f(K)$. Then for each $K\in\mathcal K$, $\overline{\{y\}}\cap f(K)\neq\emptyset$, that is, $K\cap f^{-1}\left(\overline{\{y\}}\right)\neq\emptyset$. Since $X$ is well-filtered, $f^{-1}\left(\overline{\{y\}}\right)\cap\bigcap\mathcal K\neq\emptyset$ (otherwise, $\bigcap\mathcal K\subseteq X\setminus f^{-1}\left(\overline{\{y\}}\right)$, which implies that $K\subseteq X\setminus f^{-1}\left(\overline{\{y\}}\right)$ for some $K\in\mathcal K$, a contradiction). It follows that $\overline{\{y\}}\cap f\left(\bigcap\mathcal K\right)\neq\emptyset$. This implies that $y\in\ua f\left(\bigcap\mathcal K\right)$. So $\bigcap\limits_{K\in\mathcal K}\ua f(K)\subseteq \ua f\left(\bigcap\mathcal K\right)$.

(2) $\Rightarrow$ (3)$\Rightarrow$ (4): Trivial.

(4) $\Rightarrow$ (1):  Let $\eta_X : X \rightarrow X^s$ ($=P_H(\ir_c(X))$) be the canonical topological embedding from $X$ into its soberification.  Suppose that $\mathcal K\subseteq \mathord{\mathsf K}(X)$ is filtered, $U\in \mathcal O(X)$, and $\bigcap \mathcal K \subseteq U$. If $K\not\subseteq U$ for each $K\in \mathcal K$, then by Lemma \ref{t Rudin}, $X\setminus U$ contains a minimal irreducible closed subset $A$ that still meets all members of $\mathcal{K}$. By condition (4) we have $\ua_{\ir_c(X)} \eta_X\left(\bigcap\mathcal K\right)=\bigcap\limits_{K\in\mathcal K}\ua_{\ir_c(X)} \eta_X(K)\subseteq \ua_{\ir_c(X)}\eta_X(U)=\Diamond_{\ir_c(X)} U$. Clearly, $A\in \bigcap\limits_{K\in\mathcal K}\ua_{\ir_c(X)} \eta_X(K)$, and whence $A\in \Diamond_{\ir_c(X)} U$, that is, $A\cap U\neq\emptyset$, being in contradiction with $A\subseteq X\setminus U$. Thus $X$ is well-filtered.
\end{proof}

In the above theorem, we can let $\mathbf{K}$ be the category of all $d$-spaces or that of all well-filtered spaces.

\begin{lemma}\label{Kmeet}\emph{(\cite{redbook})} For a nonempty family $\{K_i : i\in I\}\subseteq \mk (X)$, $\bigvee_{i\in I} K_i$ exists in $\mk (X)$ if{}f~$\bigcap_{i\in I} K_i\in \mk (X)$. In this case $\bigvee_{i\in I} K_i=\bigcap_{i\in I} K_i$.
\end{lemma}

 For the well-filteredness of Smyth power space, we now prove a similar result to that of sobriety in Theorem \ref{Smyth-sober}. The following result has been first proved in \cite{xuxizhao}. The proof we present here is more simple.

\begin{theorem}\label{Smythwf}
	For a $T_0$ space, the following conditions are equivalent:
\begin{enumerate}[\rm (1)]
		\item $X$ is well-filtered.
        \item $P_S(X)$ is a $d$-space.
        \item $P_S(X)$ is well-filtered.
\end{enumerate}
\end{theorem}
\begin{proof}

(1) $\Rightarrow$ (2): Suppose that $X$ is a well-filtered space. Then by Lemma \ref{Kmeet}, $\mk (X)$ is a dcpo, and $\Box U\in \sigma (\mk (X))$ for any $U\in O(X)$. Thus $P_S (X)$ is a $d$-space.

(2) $\Rightarrow$ (3): Suppose that $\{\mathcal K_d : d\in D\}\subseteq \mk(P_S(X))$ is filtered, $\mathcal U\in \mathcal O(P_S(X))$, and $\bigcap\limits_{d\in D} \mathcal K_d \subseteq \mathcal U$. If $\mathcal K_d\not\subseteq \mathcal U$ for each $d\in D$, then by Lemma \ref{t Rudin}, $\mk (X)\setminus \mathcal U$ contains an irreducible closed subset $\mathcal A$ that still meets all $\mathcal K_d$ ($d\in D$). For each $d\in D$, let $K_d=\bigcup \mathcal \ua_{\mk (X)} (\mathcal A\bigcap \mathcal K_d)$ ($=\bigcup (\mathcal A\bigcap \mathcal K_d$)). Then by Lemma \ref{K union}, $\{K_d : d\in D\}\subseteq \mk (X)$ is filtered, and $K_d\in \mathcal A$ for all $d\in D$ since $\mathcal A=\da_{\mk (X)}\mathcal A$. Let $K=\bigcap\limits_{d\in D} K_d$. Then $K\in \mk (X)$ and $K=\bigvee_{\mk (X)} \{K_d : d\in D\}\in \mathcal A$ by Lemma \ref{Kmeet} and condition (2). We claim that $K\in \bigcap\limits_{d\in D}\mathcal K_d$. Suppose, on the contrary, that $K\not\in \bigcap\limits_{d\in D}\bigcap \mathcal K_d$. Then there is a $d_0\in D$ such that $K\not\in \mathcal K_{d_0}$. Select a $G\in \mathcal A\bigcap \mathcal K_{d_0}$. Then $K\not\subseteq G$, and hence there is a $g\in K\setminus G$. It follows that $g\in K_d=\bigcup (\mathcal A\bigcap \mathcal K_d)$ for all $d\in D$ and $G\not\in \Diamond_{\mk (K)}\overline{\{g\}}$. Thus $\Diamond_{\mk (K)}\overline{\{g\}}\bigcap\mathcal A\bigcap \mathcal K_d\neq\emptyset$ for all $d\in D$. By the minimality of $\mathcal A$, we have $\mathcal A=\Diamond_{\mk (K)}\overline{\{g\}}\bigcap\mathcal A$, and consequently, $G\in \mathcal A\bigcap \mathcal K_{d_o}=\Diamond_{\mk (K)}\overline{\{g\}}\bigcap\mathcal A\bigcap \mathcal K_{d_0}$, which is a contradiction with $G\not\in \Diamond_{\mk (K)}\overline{\{g\}}$. Thus $K\in \bigcap\limits_{d\in D}\mathcal K_d\subseteq \mathcal U\subseteq \mk (X)\setminus \mathcal A$, being a contradiction with $K\in \mathcal A$. Therefore, $P_S(X)$  is well-filtered.

(3) $\Rightarrow$ (1): Suppose that $\mathcal K\subseteq \mathord{\mathsf K}(X)$ is filtered, $U\in \mathcal O(X)$, and $\bigcap \mathcal K \subseteq U$. Let $\widetilde{\mathcal K}=\{\ua_{\mk (X)}K : K\in \mathcal K\}$. Then $\widetilde{\mathcal K}\subseteq \mk (P_S(X))$ is filtered and $\bigcap \widetilde{\mathcal K} \subseteq \Box U$. By the well-filteredness of $P_S(X)$, there is a $K\in \mathcal K$ such that $\ua_{\mk (X)}K\subseteq \Box U$, and whence $K\subseteq U$, proving that $X$ is well-filtered.
\end{proof}

\begin{definition}\label{weakd-space} A $T_0$ space $X$ is said to have \emph{filtered intersection property}, $\mathbf{FTIP}$ for short, if $\bigcap \mathcal K\neq\emptyset$ for each filtered family $\mathcal K\subseteq \mk (X)$. $X$ is said to have \emph{irreducible intersection property}, $\mathbf{RIP}$ for short,  if $\bigcap \mathcal A\neq\emptyset$ for each irreducible subset $\mathcal A$ of $P_S(X)$.
\end{definition}

 By Remark \ref{meet-in-Smyth}, $X$ has $\mathbf{RIP}$ if{}f $\bigcap \mathcal A\neq\emptyset$ for all irreducible closed subset $\mathcal A$ of $P_S(X)$. For a $T_0$ space $X$, by Theorem \ref{soberequiv} and Lemma \ref{Kmeet}, we have the following implications:
\begin{center}
sobriety $\Rightarrow$ irreducible completeness of $P_S(X)$ $\Rightarrow$ $\mathbf{RIP}$  $\Rightarrow$ $\mathbf{FTIP}$;\\

sobriety $\Rightarrow$ well-filteredness $\Rightarrow$ monotone convergence  $\Rightarrow$ direct completeness of $\mk (X)$ $\Rightarrow$ $\mathbf{FTIP}$.

\end{center}

\begin{theorem}\label{WFequat2} For a $T_0$ space $X$, the following conditions are equivalent:
\begin{enumerate}[\rm (1)]
		\item $X$ is well-filtered.
        \item $\mk(X)$ is a dcpo, and $\ua \left(A\cap\bigcap \mathcal K\right)=\bigcap\limits_{K\in\mathcal K}\ua (A\cap K)$ for every filtered family $\mathcal K\subseteq \mk(X)$ and $A\in \mathcal C(X)$.
        \item $\mk(X)$ is a dcpo, and $\ua \left(A\cap\bigcap \mathcal K\right)=\bigcap\limits_{K\in\mathcal K}\ua (A\cap K)$ for every filtered family $\mathcal K\subseteq \mk(X)$ and $A\in \ir_c(X)$.
        \item $X$ has $\mathbf{FTIP}$, and $\ua \left(A\cap\bigcap \mathcal K\right)=\bigcap\limits_{K\in\mathcal K}\ua (A\cap K)$ for every filtered family $\mathcal K\subseteq \mk(X)$ and $A\in \mathcal C(X)$.
        \item $X$ has $\mathbf{FTIP}$ , and $\ua \left(A\cap\bigcap \mathcal K\right)=\bigcap\limits_{K\in\mathcal K}\ua (A\cap K)$ for every filtered family $\mathcal K\subseteq \mk(X)$ and $A\in \ir_c(X)$.
\end{enumerate}
\end{theorem}
\begin{proof} We directly have (2) $\Rightarrow$ (3) $\Rightarrow$ (5) and (2) $\Rightarrow$ (4) $\Rightarrow$ (5).

(1) $\Rightarrow$ (2): By Theorem \ref{Smythwf}, $\mk(X)$ is a dcpo. Suppose $\mathcal K\subseteq \mk(X)$ and $A\in \mathcal C(X)$. Obviously, $\ua \left(A\cap\bigcap \mathcal K\right)\subseteq \bigcap\limits_{K\in\mathcal K}\ua (A\cap K)$. On the other hand, if $x\in \bigcap\limits_{K\in\mathcal K}\ua (A\cap K)$, then for each $K\in \mathcal K$, $\downarrow x\cap A\cap K\neq\emptyset$, and hence $K\not\subseteq X\setminus \downarrow x\cap A$. It follows by the well-filteredness of $X$ that $\bigcap \mathcal K\not\subseteq X\setminus \downarrow x\cap A$, that is, $\downarrow x\cap A\cap\bigcap \mathcal K\neq\emptyset$. Therefore $x\in \ua \left(A\cap\bigcap \mathcal K\right)$. The equation $\ua \left(A\cap\bigcap \mathcal K\right)=\bigcap\limits_{K\in\mathcal K}\ua (A\cap K)$ thus holds.

(5) $\Rightarrow$ (1): Suppose that $\mathcal K\subseteq \mathord{\mathsf K}(X)$ is filtered, $U\in \mathcal O(X)$, and $\bigcap \mathcal K \subseteq U$. If $K\not\subseteq U$ for each $K\in \mathcal K$, then by Lemma \ref{t Rudin}, $X\setminus U$ contains a minimal irreducible closed subset $A$ that still meets all members of $\mathcal{K}$. By condition (5) we have $\emptyset \neq \bigcap\limits_{K\in\mathcal K}\ua (A\cap K)=\ua \left(A\cap\bigcap \mathcal K\right)=\emptyset$ since $\bigcap \mathcal K \subseteq U\subseteq X\setminus A$, a contradiction. Thus $X$ is well-filtered.

\end{proof}

By Lemma \ref{COMPminimalset} and Theorem \ref{WFequat2}, we get the following corollary.

\begin{corollary}\label{WFmincorollary}\emph{(\cite{wu-xi-xu-zhao-19})} Let $X$ be a well-filtered space and $\mathcal K\subseteq \mathord{\mathsf K}(X)$ a filtered family. Then $C=\bigcap \mathcal K\in \mk (X)$, and for each $c\in min(C)$, $\bigcap\limits_{K\in \mathcal K} \ua (\da c\cap K)=\ua \left(\da c\cap\bigcap \mathcal K\right)=\ua c$.
\end{corollary}

\begin{proposition}\label{sobercharac1} For a $T_0$ space $X$, the following conditions are equivalent:
\begin{enumerate}[\rm (1) ]
	        \item $X$ is a sober space.

            \item  For any $A\in \mathcal \ir (X)$, $\overline{A}\cap\bigcap\limits_{a\in A}\ua a\neq\emptyset$.
            \item  For any $A\in \ir_c (X)$, $A\cap\bigcap\limits_{a\in A}\ua a\neq\emptyset$.

            \item  For any $A\in \ir (X)$ and $U\in \mathcal O(X)$, $\bigcap\limits_{a\in A}\ua a\subseteq U$ implies $\ua a \subseteq U$ \emph{(}i.e., $a\in U$\emph{)} for some $a\in A$.
            \item  For any $A\in \ir_c (X)$ and $U\in \mathcal O(X)$, $\bigcap\limits_{a\in A}\ua a\subseteq U$ implies $\ua a \subseteq U$ \emph{(}i.e., $a\in U$\emph{)} for some $a\in A$.

            \item  For any $\mathcal A\subseteq \ir (P_S(X))$ with $\mathcal A\subseteq \mathcal S^u(X)$ and $U\in \mathcal O(X)$, $\bigcap \mathcal A\subseteq U$ implies $K\subseteq U$ for some $K\in\mathcal A$.
            \item  For any $\mathcal A\subseteq \ir_c (P_S(\mathcal S^u(X))$ and $U\in \mathcal O(X)$, $\bigcap \mathcal A\subseteq U$ implies $K\subseteq U$ for some $K\in\mathcal A$.

\end{enumerate}
\end{proposition}
\begin{proof} (1) $\Rightarrow$ (2): If $X$ is sober and $A\in \mathcal \ir c(X)$, then there is an $x\in X$ such that $\overline{A}=\overline{\{x\}}=\da x$, and whence $x\in \overline{A}\cap \bigcap\limits_{a\in A}\ua a$.

(2) $\Leftrightarrow$ (3): Clearly, we have (2) $\Rightarrow$ (3). Conversely, if condition (3) is satisfied, then for $A\in \ir (X)$, $\overline{A}\in \ir_c(X)$ by Lemma \ref{irrsubspace}, and $\emptyset\neq\overline{A}\cap \bigcap\limits_{b\in \overline{A}}\ua b=\overline{A}\cap \bigcap\limits_{a\in A}\ua a$ by Remark \ref{Adelta=clAdelta}.

(2) $\Rightarrow$ (4): If $\ua a\not\subseteq U$ for each $a\in A$, then $A\subseteq X\setminus U$, and hence $\overline{A}\subseteq X\setminus U$. By condition (2), $\emptyset\neq \overline{A}\cap \bigcap\limits_{a\in A}\ua a\subseteq (X\setminus U)\cap U=\emptyset$, a contradiction.

(4) $\Leftrightarrow$ (5): Obviously, (4) $\Rightarrow$ (5). Conversely, if condition (5) is satisfied, then for $A\in \ir (X)$ and $U\in \mathcal O(X)$ with $\bigcap\limits_{a\in A}\ua a\subseteq U$, we have $\overline{A}\in \ir_c (X)$ and $\bigcap\limits_{b\in \overline{A}}\ua b=\bigcap\limits_{a\in A}\ua a\subseteq U$ by Remark \ref{irrsubspace} and Lemma \ref{Adelta=clAdelta}. By condition (5), $b\in U$ for some $b\in \overline{A}$, and whence $A\cap U\neq\emptyset$. Condition (4) is thus satisfied.

(4) $\Leftrightarrow$ (6) and (5) $\Leftrightarrow$ (7): By Remark \ref{X-Smyth-irr}.

(5) $\Rightarrow$ (1): Suppose $A\in \ir_c(X)$. Then $A\cap \bigcap\limits_{a\in A}\ua a\neq\emptyset$ (otherwise, by condition (5), $A\cap \bigcap\limits_{a\in A}\ua a=\emptyset \Rightarrow \bigcap\limits_{a\in A}\ua a\subseteq X\setminus A \Rightarrow \ua a\subseteq X\setminus A$ for some $a\in A$, a contradiction). Select an $x\in A\cap \bigcap\limits_{a\in A}\ua a$. Then $A\subseteq \da x=\overline{\{x\}}\subseteq A$, and hence $A=\overline{\{x\}}$. Thus $X$ is sober.

\end{proof}

The single most important result about sober spaces is the Hofmann-Mislove Theorem (see \cite{Hofmann-Mislove} or \cite[Theorem II-1.20 and Theorem II-1.21]{redbook}).

\begin{theorem}\label{Hofmann-Mislove theorem} \emph{(The Hofmann-Mislove Theorem)} For a $T_0$ space $X$, the following conditions are equivalent:
\begin{enumerate}[\rm (1)]
            \item $X$ is a sober space.
            \item  For any $\mathcal F\in \mathrm{OFilt}(\mathcal O(X))$, there is a $K\in \mk (X)$ such that $\mathcal F=\Phi (K)$.
            \item  For any $\mathcal F\in \mathrm{OFilt}(\mathcal O(X))$, $\mathcal F=\Phi (\bigcap \mathcal F)$.
\end{enumerate}
\end{theorem}

\begin{lemma}\label{irr-induced-opne filter} Suppose that $X$ is a $T_0$ space and $\mathcal A\in \ir (P_S(X))$. Then $\mathcal F_{\mathcal A}=\bigcup_{K\in \mathcal A} \Phi (K)\in \mathrm{OFilt}(\mathcal O(X))$.
\end{lemma}
\begin{proof} Clearly, $\mathcal F_{\mathcal A}\in \sigma (\mathcal O(X))$ since $\Phi (K)\in \sigma (\mathcal O(X))$ for all $K\in \mk (X)$. Now we show that $\mathcal F_{\mathcal A}\in \mathrm{Filt}(\mathcal O(X))$. Suppose $U, V\in \mathcal F_{\mathcal A}$. Then $\mathcal A\bigcap\Box U\neq\emptyset$ and $\mathcal A\bigcap\Box V\neq\emptyset$, and hence $\mathcal A\bigcap \Box (U\cap V)=\mathcal A\bigcap \Box U\bigcap\Box V\neq\emptyset$ by $\mathcal A\in \ir (P_S(X))$. Therefore, $U\cap V\in \mathcal F_{\mathcal A}$.
\end{proof}

\begin{remark}\label{K-union-closure} For a $T_0$ space $X$ and $\mathcal A\in \ir (P_S(X))$, $\mathcal F_{\mathcal A}=\mathcal F_{\mathrm{cl} \mathcal A}$. In fact, if $U\in \mathcal O(X)$ and $U\in \mathcal F_{\mathrm{cl} \mathcal A}$, then $\mathrm{cl}\mathcal A\bigcap \Box U\neq\emptyset$, and whence $\mathcal A\bigcap \Box U\neq\emptyset$. It follows $U\in \mathcal F_{\mathcal A}$.
\end{remark}

Using the Hofman-Mislove Theorem and Lemma \ref{irr-induced-opne filter}, we present an alternative proof of the following result of Heckmann and Keimel.

\begin{theorem}\label{Smyth-sober}\emph{(\cite{Klause-Heckmann})} For a $T_0$ space $X$, the following conditions are equivalent:
\begin{enumerate}[\rm (1)]
            \item $X$ is a sober space.
            \item  For any $\mathcal A\subseteq \ir (P_S(X))$ and $U\in \mathcal O(X)$, $\bigcap \mathcal A\subseteq U$ implies $K\subseteq U$ for some $K\in\mathcal A$.
            \item  For any $\mathcal A\subseteq \ir_c (P_S(X))$ and $U\in \mathcal O(X)$, $\bigcap \mathcal A\subseteq U$ implies $K\subseteq U$ for some $K\in\mathcal A$.
            \item $P_S(X)$ is sober.
\end{enumerate}
\end{theorem}
\begin{proof} (1) $\Rightarrow$ (2): By Lemma \ref{irr-induced-opne filter}, $\mathcal F_{\mathcal A}\in \mathrm{OFlit}(\mathcal O(X))$, and hence by the Hofmann-Mislove Theorem, $\mathcal F_{\mathcal A}=\Phi (\bigcap \mathcal F_{\mathcal A})=\Phi (\bigcap \mathcal A)$. Therefore, $K\subseteq U$ for some $K\in\mathcal A$.

(2) $\Leftrightarrow$ (3): By Remark \ref{meet-in-Smyth} and Remark \ref{K-union-closure}.

(3) $\Rightarrow$ (4): Suppose $\mathcal A\subseteq \ir_c (P_S(X))$. Let $H=\bigcap \mathcal A$. Then $H\neq\emptyset$ by condition (3). Now we prove that $K\in \mk (X)$. If $\{U_i : i\in I\}\subseteq \mathcal O(X)$ such that $H\subseteq \bigcup\limits_{i\in I} U_i$, then by condition (3), there is a $K\in \mathcal A$ such that $K\subseteq  \bigcup\limits_{i\in I} U_i$. Since $K\in \mk (X)$, there is a $J\in I^{(<\omega)}$ such that $\bigcup\limits_{i\in J} U_i \supseteq K\supseteq H$. Thus $H\in \mk (X)$. For each $U\in \mathcal O(X)$, by condition (3), we have that $H\in \Box U \Leftrightarrow \mathcal A\bigcap \Box U\neq\emptyset$, proving $\mathcal A=\overline{\{H\}}$. Thus $P_S(X)$ is sober.

(4) $\Rightarrow$ (1): For any $A\in \ir (X)$ and $U\in \mathcal O(X)$ with $\bigcap\limits_{a\in A}\ua a\subseteq U$, $\xi_X (A)\in \ir (P_S(X))$ and $\bigcap\limits_{a\in A}\ua_{\mk (X)}\xi_X(a)\subseteq \Box U$. By Proposition \ref{sobercharac1}, $\ua_{\mk (X)}\xi_X (a)\subseteq \Box U$, and hence $a\in U$. By Proposition \ref{sobercharac1} again, $X$ is sober.

\end{proof}

\begin{definition}\label{irrbound} A $T_0$ space $X$ is called \emph{irreducible bounded}, $r$-\emph{bounded} for short, if for any $A\in \ir (X)$, $A$ has an upper bound in $X$, that is, there is an $x\in X$ such that $A\subseteq \da x=\overline {\{x\}}$, or equivalently, $A^{\ua}=\bigcap_{a\in A}\ua d\neq\emptyset$.
\end{definition}

By Remark \ref{Adelta=clAdelta}, $X$ is $r$-bounded if{}f $A$ has an upper bound in $X$ for each $A\in \ir (X)$. Clearly, we have the following implications:

\begin{center}
sobriety $\Rightarrow$ irreducible completeness $\Rightarrow$ $r$-boundedness.
\end{center}

For a poset $P$ with a largest element $\top$, any order compatible topology $\tau$ on $P$ is $r$-bounded.

\begin{proposition}\label{sobercharac2} For a $T_0$ space $X$, the following conditions are equivalent:
\begin{enumerate}[\rm (1) ]
	        \item $X$ is sober.
            \item  $X$ is $r$-bounded \emph{(}especially, $X$ is $r$-complete\emph{)}, and $\ua \left(C\cap\bigcap\limits_{a\in A} \ua a\right)=\bigcap\limits_{a\in A}\ua (C\cap \ua a)$ for any $A\in \ir (X)$ and $C\in \mathcal C(X)$.
            \item  $X$ is $r$-bounded \emph{(}especially, $X$ is $r$-complete\emph{)}, and $\ua \left(C\cap\bigcap\limits_{a\in A} \ua a\right)=\bigcap\limits_{a\in A}\ua (C\cap \ua a)$ for any $A\in \ir (X)$ and $C\in \ir_c (X)$.
            \item  $X$ is $r$-bounded \emph{(}especially, $X$ is $r$-complete\emph{)}, and $\ua \left(C\cap\bigcap \mathcal K\right)=\bigcap\limits_{K\in\mathcal K}\ua (C\cap K)$ for any $\mathcal A\in \ir (P_S(X))$ with $\mathcal A\subseteq\mathcal S^u(X)$ and $C\in \mathcal C(X)$.
            \item  $X$ is $r$-bounded \emph{(}especially, $X$ is $r$-complete\emph{)}, and $\ua \left(C\cap\bigcap \mathcal K\right)=\bigcap\limits_{K\in\mathcal K}\ua (C\cap K)$ for any $\mathcal A\in \ir (P_S(X))$ with $\mathcal A\subseteq\mathcal S^u(X)$ and $C\in \ir_c(X)$.
\end{enumerate}
\end{proposition}
\begin{proof}  (1) $\Rightarrow$ (2): Since $X$ is sober, $X$ is $r$-complete by Remark \ref{Soberirrcomp}. For $A\in \ir (X)$ and $C\in \mathcal C(X)$, clearly, $\ua \left(C\cap\bigcap\limits_{a\in A} \ua a\right)\subseteq\bigcap\limits_{a\in A}\ua (C\cap \ua a)$. Conversely, if $x\not\in \ua \left(C\cap\bigcap\limits_{a\in A}\ua a\right)$, that is, $\da x\cap C\cap\bigcap\limits_{a\in A}\ua a=\emptyset$, then $\bigcap\limits_{a\in A}\ua a\subseteq X\setminus \da x\cap C$, and whence by Theorem \ref{sobercharac1}, $a\in X\setminus \da x\cap C$ for some $a\in A$, i.e., $x\not\in \ua (C\cap\ua a)$. Therefore, $x\not\in \bigcap\limits_{a\in A}\ua (C\cap \ua a)$. Thus $\ua \left(C\cap\bigcap\limits_{a\in A} \ua d\right)=\bigcap\limits_{a\in A}\ua (C\cap \ua a)$.

(2) $\Rightarrow$ (3): Trivial.

(2) $\Leftrightarrow$ (4) and (3) $\Leftrightarrow$ (5): For $\mathcal A=\{\ua x_i : i\in I\}\subseteq \mathcal S^u(X)$, $\mathcal A\in \ir (P_S(X))$
if{}f $\{x_i : i\in I\}\in \ir (X)$.

(3) $\Rightarrow$ (1): For each $A\in \ir_c(X)$, by condition (3), $\emptyset \neq A^{\ua}=\bigcap\limits_{a\in A}\ua a=\bigcap\limits_{a\in A}\ua (A\cap \ua a)=\ua \left(A\cap\bigcap\limits_{a\in A} \ua d\right)$. By Theorem \ref{sobercharac1}, $X$ is sober.
\end{proof}

\begin{remark}\label{sobercharac2-1} For a $T_0$ space $X$, by the proof of Proposition \ref{sobercharac2}, the following conditions are equivalent:
\begin{enumerate}[\rm (a) ]
	        \item $X$ is sober.
            \item  $X$ is $r$-bounded \emph{(}especially, $X$ is $r$-complete\emph{)}, and $\ua \left(C\cap\bigcap\limits_{a\in A} \ua a\right)=\bigcap\limits_{a\in A}\ua (C\cap \ua a)$ for any $A\in \ir_c (X)$ and $C\in \mathcal C(X)$.
            \item  $X$ is $r$-bounded \emph{(}especially, $X$ is $r$-complete\emph{)}, and $\ua \left(C\cap\bigcap\limits_{a\in A} \ua a\right)=\bigcap\limits_{a\in A}\ua (C\cap \ua a)$ for any $A\in \ir_c (X)$ and $C\in \ir_c (X)$.
\end{enumerate}
\end{remark}

\begin{theorem}\label{sobercharact3} For a $T_0$ space $X$ , the following conditions are equivalent:
\begin{enumerate}[\rm (1)]
		\item $X$ is sober.
        \item $X$ has $\mathbf{\mathrm{RIP}}$ (especially, $P_S(X)$ is $r$-complete) and $\ua \left(C\cap\bigcap \mathcal K\right)=\bigcap\limits_{K\in\mathcal K}\ua (C\cap K)$ for every $\mathcal A\subseteq \ir(P_S(X))$ and $C\in \mathcal C(X)$.
        \item $X$ has $\mathbf{\mathrm{RIP}}$ (especially, $P_S(X)$ is $r$-complete) and $\ua \left(C\cap\bigcap \mathcal K\right)=\bigcap\limits_{K\in\mathcal K}\ua (C\cap K)$ for every $\mathcal A\subseteq \ir(P_S(X))$ and $C\in \ir_c(X)$.
        \item $X$ has $\mathbf{\mathrm{RIP}}$ (especially, $P_S(X)$ is $r$-complete) and $\ua \left(C\cap\bigcap \mathcal K\right)=\bigcap\limits_{K\in\mathcal K}\ua (C\cap K)$ for every $\mathcal A\subseteq \ir_c(P_S(X))$ and $C\in \mathcal C(X)$.
        \item $X$ has $\mathbf{\mathrm{RIP}}$ (especially, $P_S(X)$ is $r$-complete) and $\ua \left(C\cap\bigcap \mathcal K\right)=\bigcap\limits_{K\in\mathcal K}\ua (C\cap K)$ for every $\mathcal A\subseteq \ir_c(P_S(X))$ and $C\in \ir_c(X)$.
\end{enumerate}
\end{theorem}
\begin{proof} We directly have (2) $\Rightarrow$ (3) $\Rightarrow$ (5) and  (2) $\Rightarrow$ (4) $\Rightarrow$ (5).

(1) $\Rightarrow$ (2): By Remark \ref{Soberirrcomp} and Corollary \ref{Smyth-sober}, $X$ is $r$-complete. Suppose $\mathcal A\subseteq \ir(P_S(X))$ and $C\in \mathcal C(X)$. Obviously, $\ua \left(C\cap\bigcap \mathcal A\right)\subseteq\bigcap_{K\in\mathcal A}\ua (C\cap K)$. On the other hand, if $x\in \bigcap_{K\in\mathcal A}\ua (C\cap K)$, then for each $K\in \mathcal A$, $\downarrow x\cap C\cap K\neq\emptyset$, and hence $K\not\subseteq X\setminus \downarrow x\cap C$. By Corollary \ref{Smyth-sober}, we have $\bigcap \mathcal K\not\subseteq X\setminus \downarrow x\cap C$, that is, $\downarrow x\cap C\cap\bigcap \mathcal K\neq\emptyset$. Therefore $x\in \ua \left(C\cap\bigcap \mathcal A\right)$. The equation $\ua \left(C\cap\bigcap \mathcal A\right)=\bigcap_{K\in\mathcal A}\ua (C\cap K)$ thus holds.

(5) $\Rightarrow$ (1): Suppose that $\mathcal A\subseteq \ir_c(P_S(X))$, $U\in \mathcal O(X)$, and $\bigcap \mathcal A \subseteq U$. If $K\not\subseteq U$ for each $K\in \mathcal A$, then by Lemma \ref{t Rudin}, $X\setminus U$ contains a minimal irreducible closed subset $C$ that still meets all members of $\mathcal A$. Let $\mathcal A_C=\{\ua (C\cap K) : K\in \mathcal A\}$. Then $\mathcal A_C\subseteq \mk (X)$. Now we show that $\mathcal A_C\in \ir (P_S(X))$. Suppose $V, W\in \mathcal O(X)$ such that $\mathcal A_C\bigcap \Box V\neq\emptyset$ and $\mathcal A_C\bigcap \Box W\neq\emptyset$. Then $\mathcal A\bigcap \Box (V\cup(X\setminus C))\neq\emptyset$ and $\mathcal A\bigcap \Box (W\cup(X\setminus C)\neq\emptyset$, and whence $\mathcal A\bigcap \Box ((V\cap W)\cup(X\setminus C))=\mathcal A\bigcap \Box (V\cup(X\setminus C))\bigcap \Box (W\cup(X\setminus C))\neq\emptyset$ by the irreducibility of $\mathcal A$. It follows that $\mathcal A_C\bigcap \Box V\bigcap \Box W=\mathcal A_C\bigcap \Box (V\cap W)\neq\emptyset$, proving the irreducibility of $\mathcal A_C$. By condition (5) we have $\emptyset \neq \bigcap \mathcal A_C=\ua \left(C\cap\bigcap \mathcal K\right)=\emptyset$ since $\bigcap \mathcal A \subseteq U\subseteq X\setminus C$, a contradiction. Thus $X$ is sober by Theorem \ref{Smyth-sober}.

\end{proof}

As a corollary of Theorem \ref{sobercharact3}, we get a similar result to Corollary \ref{WFmincorollary}.

\begin{corollary}\label{Sobermincorollary}  Let $X$ be a sober space and $\mathcal A\subseteq \ir (P_S(X))$. Then $C=\bigcap \mathcal A\in \mk (X)$, and for each $c\in min(C)$, $\bigcap_{K\in \mathcal A} \ua (\da c\cap K)=\ua \left(\da c\cap\bigcap \mathcal A\right)=\ua c$.
\end{corollary}

\begin{theorem}\label{soberequat4} Let $X$ be a $T_0$ space and $\mathbf{K}$ a full subcategory of $\mathbf{Top}_0$ containing $\mathbf{Sob}$. Then the following conditions are equivalent:
\begin{enumerate}[\rm (1)]
		\item $X$ is sober.
        \item For every continuous mapping $f:X\longrightarrow Y$ from $X$ to a $T_0$ space $Y$ and any $\mathcal A\in \ir (P_S(X))$, $\ua f\left(\bigcap\mathcal K\right)=\bigcap_{K\in\mathcal K}\ua f(K)$,
        \item For every continuous mapping $f:X\longrightarrow Y$ from $X$ to a $T_0$ space $Y$ and any $\mathcal A\in \ir_c (P_S(X))$, $\ua f\left(\bigcap\mathcal K\right)=\bigcap_{K\in\mathcal K}\ua f(K)$,

        \item For every continuous mapping $f:X\longrightarrow Y$ from $X$ to a $\mathbf{K}$-space $Y$ and any $\mathcal A\in \ir (P_S(X))$, $\ua f\left(\bigcap\mathcal K\right)=\bigcap_{K\in\mathcal K}\ua f(K)$.
        \item For every continuous mapping $f:X\longrightarrow Y$ from $X$ to a a $\mathbf{K}$-space $Y$ and any $\mathcal A\in \ir_c (P_S(X))$, $\ua f\left(\bigcap\mathcal K\right)=\bigcap_{K\in\mathcal K}\ua f(K)$.

        \item For every continuous mapping $f:X\longrightarrow Y$ from $X$ to a sober space $Y$ and any $\mathcal A\in \ir (P_S(X))$, $\ua f\left(\bigcap\mathcal K\right)=\bigcap_{K\in\mathcal K}\ua f(K)$.
        \item For every continuous mapping $f:X\longrightarrow Y$ from $X$ to a sober space $Y$ and any $\mathcal A\in \ir_c (P_S(X))$, $\ua f\left(\bigcap\mathcal K\right)=\bigcap_{K\in\mathcal K}\ua f(K)$.
\end{enumerate}
\end{theorem}
\begin{proof} We only need to prove the equivalences of conditions (1), (2), (3), (6), and (7).

(1) $\Rightarrow$ (2): It needs only to check $\bigcap_{K\in\mathcal A}\ua f(K)\subseteq \ua f\left(\bigcap\mathcal A\right)$. Let $y\in\bigcap_{K\in\mathcal A}\ua f(K)$. Then for each $K\in\mathcal A$, $\overline{\{y\}}\cap f(K)\neq\emptyset$, that is, $K\cap f^{-1}\left(\overline{\{y\}}\right)\neq\emptyset$. Since $X$ is sober, $f^{-1}\left(\overline{\{y\}}\right)\cap\bigcap\mathcal A\neq\emptyset$ (otherwise, $\bigcap\mathcal A\subseteq X\setminus f^{-1}\left(\overline{\{y\}}\right)$, and consequently,  by Theorem \ref{Smyth-sober},  $K\subseteq X\setminus f^{-1}\left(\overline{\{y\}}\right)$ for some $K\in\mathcal A$, a contradiction). It follows that $\overline{\{y\}}\cap f\left(\bigcap\mathcal A\right)\neq\emptyset$. This implies that $y\in\ua f\left(\bigcap\mathcal K\right)$. So $\bigcap_{K\in\mathcal A}\ua f(K)\subseteq \ua f\left(\bigcap\mathcal A\right)$.

(2) $\Rightarrow$ (3), (2) $\Rightarrow$ (6),(3) $\Rightarrow$ (7) and (6) $\Rightarrow$ (7): Trivial.

(7) $\Rightarrow$ (1):  Let $\eta_X : X \rightarrow X^s$ ($=P_H(\ir_c(X))$) be the canonical topological embedding from $X$ into its soberification and $\xi_X : X \rightarrow P_S(X)$ the canonical topological embedding from $X$ into the Smyth power space of $X$.  Suppose that $\mathcal A\subseteq \ir_c (X)$. Then $\Diamond_{\mk (X)} A=\mathrm{cl}_{P_S(X)}{\xi_X(A)}\in \ir_c (P_S(X))$. By Remark \ref{meet-in-Smyth} and condition (7), we have $\ua_{\ir_c(X)}\eta_X(A^{\ua})=\ua_{\ir_c(X)}\eta_X(\bigcap \Diamond_{\mk (X)} A)=\bigcap\limits_{K\in \Diamond_{\mk (X)} A} \ua_{\ir_c(X)}\eta_X(K)=\bigcap\limits_{a\in A} \ua_{\ir_c(X)} \overline{\{a\}}$. It follows that $A\in \ua_{\ir_c(X)}\eta_X(A^{\ua})$ by $A\in \bigcap\limits_{a\in A} \ua_{\ir_c(X)} \overline{\{a\}}$. Therefore, there is an $x\in A^{\ua}$ such that $\overline{\{x\}}\subseteq A$, and consequently, $A=\overline{\{x\}}$. Thus $X$ is sober.

\end{proof}

\section{Well-filtered determined spaces}

In this section, we introduce another  new type of subsets in a $T_0$ topological space - well-filtered
determined sets ($\wdd$ sets for short), which is closed related to Rudin sets. Using $\wdd$ sets, we introduce and investigate another new kind of spaces - well-filtered determined spaces ($\wdd$ spaces for short). The Rudin spaces lie between $wdd$ spaces and $DC$ spaces, and $DC$ spaces lie between Rudin spaces and sober spaces. For a $T_0$ space $X$, it is proved that $X$ is sober if{}f  $X$ is a well-filtered Rudin space if{}f $X$ is a well-filtered $\wdd$ space.

 In \cite{E_20182}, it is shown that in a locally hypercompact $T_0$ space $X$, every irreducible closed subset $A$ of $X$ is the closure of a certain directed subset of $X$. Therefore, locally hypercompact spaces are $\mathbf{DC}$ spaces. Further, we prove that every locally compact $T_0$ space is a  Rudin space and every core compact $T_0$ space is a $\wdd$ space. As a corollary we have that every core compact well-filtered space is sober, giving a positive answer to Jia-Jung problem \cite{jia-2018}, which has been independently given by Lawson and Xi \cite{Lawson-Xi} in a different way.

Firstly, motivated by Proposition \ref{rudinwf}, we give the following definition.

\begin{definition}\label{WDspace}
	 A subset $A$ of a $T_0$ space $X$ is called a \emph{well-filtered determined set}, $\wdd$ \emph{set} for short, if for any continuous mapping $ f:X\longrightarrow Y$
to a well-filtered space $Y$, there exists a unique $y_A\in Y$ such that $\overline{f(A)}=\overline{\{y_A\}}$.
Denote by $\mathsf{WD}(X)$ the set of all closed well-filtered determined subsets of $X$. $X$ is called a \emph{well-filtered determined}, $\mathsf{WD}$ \emph{space} for short, if all irreducible closed subsets of $X$ are well-filtered determined, that is, $\ir_c(X)=\wdd (X)$.
\end{definition}

Obviously, a subset $A$ of a space $X$ is well-filtered determined if{}f $\overline{A}$ is well-filtered determined.

\begin{proposition}\label{DRWIsetrelation}
	Let $X$ be a $T_0$ space. Then $\mathcal{D}_c(X)\subseteq \mathsf{RD}(X)\subseteq\mathsf{WD}(X)\subseteq\ir_c(X)$.
\end{proposition}
\begin{proof}
	By Lemma \ref{DRIsetrelation} and Lemma \ref{rudinwf}, $\mathcal{D}_c(X)\subseteq \mathsf{RD}(X)\subseteq\mathsf{WD}(X)$. We need to show $\mathsf{WD}(X)\subseteq\ir_c(X)$. Let $A\in\mathsf{WD}(X)$.
	Since $\eta_X: X\longrightarrow X^s,\ x\mapsto\da x$, is a continuous mapping to a well-filtered space ($X^s$ is sober), there exists $C\in \ir_c(X)$ such that $\overline{\eta_X(A)}=\overline{\{C\}}$.
	Let $U\in\mathcal O(X)$.
	Note that $$\begin{array}{lll}
	A\cap U\neq\emptyset &\Leftrightarrow& \eta_X(A)\cap  \Diamond U\neq\emptyset\\
	&\Leftrightarrow&\{C\}\cap \Diamond U\neq\emptyset\\
	&\Leftrightarrow& C\in \Diamond U\\
	&\Leftrightarrow& C\cap U\neq\emptyset.
	\end{array}$$
	This implies that $A=C$, and hence $A\in \ir_c(X)$.
\end{proof}

\begin{corollary}\label{SDRWspacerelation}
	Sober $\Rightarrow$ $\mathsf{DC}$ $\Rightarrow$ $\mathsf{RD}$ $\Rightarrow$ $\mathsf{WD}$.
\end{corollary}

By Theorem \ref{WFequat1} and Corollary \ref{SDRWspacerelation}, we have the following corollary.

\begin{corollary}\label{WFequat3} For a $T_0$ space $X$, the following conditions are equivalent:
\begin{enumerate}[\rm (1)]
		\item $X$ is well-filtered.
        \item For every continuous mapping $f:X\longrightarrow Y$ from $X$ to a $\wdd$ space $Y$ and a filtered family $\mathcal K\subseteq \mk(X)$, $\ua f\left(\bigcap\mathcal K\right)=\bigcap_{K\in\mathcal K}\ua f(K)$.
        \item For every continuous mapping $f:X\longrightarrow Y$ from $X$ to a $\kf$ space $Y$ and a filtered family $\mathcal K\subseteq \mk(X)$, $\ua f\left(\bigcap\mathcal K\right)=\bigcap_{K\in\mathcal K}\ua f(K)$.
        \item For every continuous mapping $f:X\longrightarrow Y$ from $X$ to a $\dc$ space $Y$ and a filtered family $\mathcal K\subseteq \mk(X)$, $\ua f\left(\bigcap\mathcal K\right)=\bigcap_{K\in\mathcal K}\ua f(K)$.

\end{enumerate}
\end{corollary}

By \cite[Proposition 4]{well} and \cite[Theorem 5.7]{ZhaoHo}, we get the following result.

\begin{proposition}\label{alpha}
	Let $P$ be a poset. Then the Alexandroff space $(P,\alpha(P))$ is a $\mathsf{DC}$ space and the following conditions are equivalent:
	\begin{enumerate}[\rm (1)]
		\item $(P,\alpha(P))$ is sober.
		\item $(P,\alpha(P))$ is well-filtered.
		\item $(P,\alpha(P))$ is a $d$-space.
		\item $P$ satisfies the $\mathrm{ACC}$ condition;
		\item $P$ is a dcpo such that every element of $P$ is compact \emph{(}i.e., $x\ll x$ for all $x\in P$\emph{)}.
		\item $P$ is a dcpo such that $\alpha(P)=\sigma(P)$.
	\end{enumerate}
\end{proposition}

\begin{theorem}\label{soberequiv} For a $T_0$ space $X$, the following conditions are equivalent:
	\begin{enumerate}[\rm (1)]
		\item $X$ is sober.
		\item $X$ is a $\mathsf{DC}$ $d$-space.
        \item $X$ is a well-filtered $\mathsf{DC}$ space.
		\item $X$ is a well-filtered Rudin space.
		\item $X$ is a well-filtered $\mathsf{WD}$ space.
	\end{enumerate}
\end{theorem}
\begin{proof}
	By Corollary \ref{SDRWspacerelation} we only need to check (5) $\Rightarrow$ (1).
	Assume $X$ is a well-filtered $\mathsf{WD}$ space. Let $A\in\ir_c(X)$. Since the identity $id_X : X\longrightarrow X$ is continuous, there is a unique $x\in X$ such that $\overline{A}=\overline{\{x\}}$. So $X$ is sober.
\end{proof}

\begin{lemma}\label{LHCdirected} \emph{(\cite{E_20182})}
	Let $X$ be a locally hypercompact $T_0$ space and $A\in\ir(X)$. Then there exists a directed subset $D\subseteq\da A$ such that $\overline{A}=\overline{D}$.
\end{lemma}

\begin{remark}\label{LHCdirectedremark} For $C\subseteq X$, we have $C^{\delta}=\bigcap \{\downarrow x: C\subseteq \downarrow x\}=\bigcap \{\downarrow x: \overline{C}\subseteq \downarrow x\}=\overline{C}^{\delta}$. So in Lemma \ref{LHCdirected} we also have $A^{\delta}=D^{\delta}$.
\end{remark}

By Corollary \ref{SDRWspacerelation} and Lemma \ref{LHCdirected}, we get the following corollary.

\begin{corollary}\label{erne1}
	If $X$ is a locally hypercompact $T_0$ space, then it is a $\mathsf{DC}$ space. Therefore, it is a Rudin space and a $\wdd$ space.
\end{corollary}

\begin{theorem}\label{LCrudin}
	Every locally compact $T_0$ space is a Rudin space.
\end{theorem}
\begin{proof}
	Suppose that $X$ is a locally compact $T_0$ space and $A\in \ir_c(X)$. Let $\mathcal K_A=\{K \in \mk (X) : A\cap\ii \, K\neq\emptyset\}$.
	
	{Claim 1:} $\mathcal K_A\neq\emptyset$.
	
	Let $a\in A$. Since $X$ is locally compact, there exists a $K\in \mk (X)$ such that $a\in\ii \, K$. So $a\in A\cap\ii \, K$ and $K\in \mathcal K_A$.
	
	{Claim 2:} $\mathcal K_A$ is filtered.
	
	Let $K_1, K_2\in\mathcal K_A$, that is, $A\cap\ii \, K_1\neq\emptyset$ and $A\cap\ii \, K_2\neq\emptyset$. Since $A$ is irreducible, $A\cap\ii \,  F_1\cap\ii \, K_2\neq\emptyset$. Let $x\in A\cap\ii \, K_1\cap\ii \, K_2$. By the local compactness of $X$ again, there exists a $K_3\in \mk (X)$ such that $x\in \ii \, K_3\subseteq K_3\subseteq \ii \, K_1\cap\ii \, K_2$. Thus $K_3\in\mathcal K_A$ and $K_3\subseteq K_1\cap K_2$. So $\mathcal K_A$ is filtered.

    {Claim 3:}  $A\in m(\mathcal K_A)$.	

Clearly, $\mathcal K_A\subseteq \Diamond A$. If $B$ is a proper closed subset of $A$, then there is $a\in A\setminus B$. Since $X$ is locally compact, there is $K_a\in \mk (X)$ such that $a\in \ii \, K_a\subseteq K_a\subseteq X\setminus B$. Then $K_a\in \mathcal K_A$ but $K_a\cap B=\emptyset$, and whence $B\not\in M(\mathcal K_A)$, proving $A\in m(\mathcal K_A)$. Thus $X$ is a Rudin space.
\end{proof}

\begin{definition}\label{subsetWaybelow} For a $T_0$ space and $A, B\subseteq X$, we say $A$ is \emph{way below} $B$, or $A$ is \emph{compact relative to} $B$, written as $A\ll B$, if for each $\{U_i : i\in I\}\subseteq \mathcal O(X)$, $B\subseteq \bigcup_{i\in I} U_i$ implies $A\subseteq\bigcup_{i\in J} U_i$ for some finite subset $J$ of $I$.
\end{definition}

Clearly, we have $A\ll B \Rightarrow \ua A\subseteq \ua B$, and if $A, B, G, H\in \mathbf{up}(X)$, then $G\subseteq A\ll B \subseteq H\Rightarrow G\ll H$.

\begin{definition}\label{waybelowsequence} Let $X$ be a $T_0$ space and $\mathcal F=\{F_\infty, ..., F_n, ..., F_2, F_1\}\subseteq \mathbf{up}(X)$. $\mathcal F$ is called a \emph{bounded decreasing $\ll$-sequence} in $X$ if $F_\infty\ll ...\ll F_n\ll ...\ll F_2\ll F_1$. Denote the minimal set $F_\infty$ in $\mathcal F$ by $min ~\!\mathcal F$ and the maximal set $F_1$ in $\mathcal F$ by $max~\! \mathcal F$.
\end{definition}

\begin{lemma}\label{waybelowimage} Let $f:X\longrightarrow Y$ be a continuous mapping and $A, B\subseteq X$. If $A\ll B$, then $f(A)\ll f(B)$.
\end{lemma}
\begin{proof} Suppose $\{V_i : i\in I\}\subseteq \mathcal O(Y)$. If $f(B)\subseteq \bigcup_{i\in I} V_i$, then $B\subseteq f^{-1} (\bigcup_{i\in I} V_i)=\bigcup_{i\in I} f^{-1}(V_i)$. Since $f$ is continuous and $A\ll B$,  there is a $J\in I^{(<\omega)}$ such that $A\subseteq\bigcup_{i\in J} f^{-1}(V_i)$, and whence $f(A)\subseteq\bigcup_{i\in J} V_i$. Thus $f(A)\ll f(B)$.

\end{proof}

\begin{corollary}\label{d-below-sequence image} Let $X$ and $Y$ be $T_0$ spaces and $f:X\longrightarrow Y$ a continuous mapping. If $\mathcal F\subseteq \mathbf{up}(X)$ is a bounded decreasing $\ll$-sequence in $X$, then $\{\ua f(F) : F\in\mathcal F\}$ is a bounded decreasing $\ll$-sequence in $Y$.
\end{corollary}

\begin{theorem}\label{CorecomptWD} Every core compact $T_0$ space is well-filtered determined.
\end{theorem}
\begin{proof}
Let $X$ be a core compact $T_0$ space and $A\in \ir_c (X)$. We need to show $A\in\wdd (X)$. Suppose that $ f : X\longrightarrow Y$ is a continuous mapping from $X$ to a well-filtered space $Y$. Let $\mathfrak{F}_A=\{\mathcal F : \mathcal F\subseteq \mathcal O(X) \mbox{ is a bounded decreasing $\ll$-sequence in $X$ with } A\cap min ~\!\mathcal F\neq\emptyset\}$. Define a partial order $\prec$ on $\mathfrak{F}_A$ by $\mathcal F_1\prec\mathcal F_2$ if{}f $max ~\!\mathcal F_1\subseteq min ~\!\mathcal F_2$. For each $\mathcal F\in \mathfrak{F}_A$, let $K_{\mathcal F}=\bigcap\limits_{U\in \mathcal F\setminus \{min ~\!\mathcal F\}}\ua f(U)$.

{Claim 1:} $\mathfrak{F}_A\neq\emptyset$.

Select a point $a\in A$ and a $U\in \mathcal O(X)$. Then by the core compactness of $X$, there is a sequence $\mathcal F_a=\{U_\infty, ..., U_n, ..., U_2, U_1\}\subseteq \mathcal O(X)$ such that $a\in U_\infty$ and $U_\infty\ll ...\ll U_n\ll ...\ll U_2\ll U_1=U$. Then $\mathcal F_a\in \mathfrak{F}_A$.

{Claim 2:} $\mathfrak{F}_A$ is $\prec$-filtered.

Suppose that $\mathcal F_1, \mathcal F_2\in \mathfrak{F}_A$. Then $A\cap min~\!\mathcal F_1\neq\emptyset\neq A\cap min~\!\mathcal F_2$, and hence $A\cap min~\!\mathcal F_1\cap min~\!\mathcal F_2\neq\emptyset$ by the irreducibility of $A$. Let $W_1=min~\!\mathcal F_1\cap min ~\!\mathcal F_2$ and select a point $b\in A\cap W_1$. Then by the core compactness of $X$, there is a sequence $\mathcal F_3=\{W_\infty, ..., W_n, ..., W_2, W_1\}\subseteq \mathcal O(X)$ such that $b\in W_\infty$ and $W_\infty\ll ...\ll W_n\ll ...\ll W_2\ll W_1=W$. Then $\mathcal F_3\in \mathfrak{F}_A$, $\mathcal F_3\prec \mathcal F_1$ and  $\mathcal F_3\prec \mathcal F_2$.

{Claim 3:} For $\mathcal F\in \mathfrak{F}_A$ and $W\in \mathcal O(Y)$, if $K_{\mathcal F}\subseteq W$, then $\ua f(U)\subseteq W$ for some $U\in \mathcal F\setminus \{min~\! \mathcal F\}$.

Let $\mathcal F=\{U_\infty, ..., U_n, ..., U_2, U_1\}\subseteq \mathcal O(X)$ and $U_\infty\ll ...\ll U_n\ll ...\ll U_2\ll U_1$. Assume, on the contrary, that $\ua f(U_n)\not\subseteq W$ for all $n\in N$. Let $\mathcal B=\{B\in\mathcal C(Y): B\subseteq Y\setminus W \text{ and } \mathcal \ua f(U_n)\cap B\neq\emptyset \text{ for all }  n\in N\}$. Then we have the following two facts.
	
	{(b1)} $\mathcal B\neq\emptyset$ because $Y\setminus W\in\mathcal B$.
	
	{(b2)} For any filtered family $\mathcal Z\subseteq\mathcal B$, $\bigcap\mathcal Z\in\mathcal B$.
	
	Let $Z=\bigcap\mathcal Z$. Then $Z\in \mathcal C(Y)$ and $Z\subseteq Y\setminus W$. Assume $Z\not\in\mathcal B$. Then there exists $n\in N$ such that $f(U_n)\cap Z=\emptyset$. Then $U_{n+1}\ll U_n\subseteq \bigcup_{C\in \mathcal Z}f^{-1}(Y\setminus C)$, and consequently, there is a $C\in \mathcal Z$ such that $U_{n+1}\subseteq f^{-1}(Y\setminus C)$, that is, $f(U_{n+1})\cap C=\emptyset$, which is a contradiction with $C\in \mathcal Z\subseteq\mathcal B$. Therefore, $\bigcap\mathcal Z\in\mathcal B$.
	
	By Zorn's Lemma, there exists a minimal element $E$ in $\mathcal B$. Since $E=\da E$, $E$ intersects all $f(U_n)$. For each $n\in N$, select an $e_n\in f(U_n)\cap E$ and let $H_n=\{e_m : n\leq m\}$. Now we prove that $\ua H_n\in \mk (Y)$ for all $n\in N$. Suppose that $\{V_d : d\in D\}\subseteq \mathcal O(Y)$ is a directed open cover of $\ua H_n$.

     {(c1)} If for some $d_1\in D$, $H_n\cap (Y\setminus V_{d_1})=H_n\setminus V_{d_1}$ is finite, then $H_n\setminus V_{d_1}\subseteq V_{d_2}$ for some $d_2\in D$ because $H_n\subseteq \bigcup\limits_{d\in D}V_d$. By the directness of $\{V_d : d\in D\}$, $V_{d_1}\cup V_{d_1}\subseteq V_{d_3}$ for some $d_3\in D$. Then $H_n\subseteq V_{d_3}$.

      {(c2)} If for all $d\in D$, $H_n\cap (Y\setminus V_{d})$ is infinite, then $f(U_n)\cap E\cap (Y\setminus V_d)\neq\emptyset$ since $\mathcal F=\{U_\infty, ..., U_n, ..., U_2, U_1\}$ is a bounded decreasing $\ll$-sequence in $X$, and whence $E\cap (Y\setminus V_d)\in \mathcal B$. By the minimality of $E$, $E\cap (Y\setminus V_d)=E$ for all $d\in D$. Therefore, $H_n\subseteq E\subseteq \bigcap_{d\in D}(Y\setminus V_d)=Y\setminus \bigcup_{d\in D}V_d$, which is a contradiction with $H_n\subseteq \bigcup_{d\in D}V_d$.

By (c1) and (c2), $H_n\in  \mk (Y)$. Clearly, $\{\ua H_n : n\in N\}\subseteq \mk (Y)$ is filtered, and whence $H=\bigcap_{n\in N}\ua H_n\in \mk (Y)$ by the well-filteredness of $Y$. It follows that $\emptyset \neq H\subseteq E\cap \bigcap_{n\in N}\ua f(U_n)\subseteq (Y\setminus W)\cap W=\emptyset$ (note that $f(U_1)\supseteq f(U_2)\supseteq ... \supseteq f(U_n)\supseteq ...$), a contradiction, proving Claim 3.

   {Claim 4:} $K_{\mathcal F}\in \mk (Y)$ for each $\mathcal F\in \mathfrak{F}_A$.

Suppose $\{V_i : i\in I\}\subseteq \mathcal O(Y)$ and $K_{\mathcal F}\subseteq \bigcup_{i\in I} V_i$. Then by Claim 3, $\ua f(U)\subseteq \bigcup_{i\in I} V_i$ for some $U\in \mathcal F\setminus \{min~\! \mathcal F\}$, and whence $U\subseteq \bigcup_{i\in I} f^{-1}(V_i)$. Since $\mathcal F$ is a bounded decreasing $\ll$-sequence and $U\neq min~\! \mathcal F$, there is a $U^*\in \mathcal F$ such that $U^{*}\ll U$. It follows that $U^*\subseteq \bigcup_{i\in J} f^{-1}(V_i)$ for some $J\in I^{(<\omega)}$, and consequently, $K_{\mathcal F}=\bigcap\limits_{U\in \mathcal F\setminus \{min~\! \mathcal F\}}\ua f(U)\subseteq\ua f(U^*)\subseteq \bigcup_{i\in J} V_i$. Thus $K_{\mathcal F}\in \mk (Y)$.

   {Claim 5:} $\bigcap_{\mathcal F\in \mathfrak{F}_A} K_{\mathcal F}\in \mk (Y)$.

By Claim 2 and Claim 4, $\{ K_{\mathcal F} : \mathcal F\in \mathfrak{F}_A\}\in \mk (Y)$ if filtered, and whence $\bigcap_{\mathcal F\in \mathfrak{F}_A} K_{\mathcal F}\in \mk (Y)$ by the well-filteredness of $Y$.

 {Claim 6:} $A\in \wdd (X)$.

  We first show that $\bigcap\limits_{\mathcal F\in \mathfrak{F}_A} K_{\mathcal F}\cap\overline{f(A)} \neq\emptyset$. Assume, on the contrary, $\bigcap\limits_{\mathcal F\in \mathfrak{F}_A} K_{\mathcal F} \subseteq Y\setminus \overline{f(A)}$, then by Claim 3, Claim 5 (and its proof) and the well-filteredness of $Y$, there is an $\mathcal F\in \mathfrak{F}_A\}$ such that $\ua f(U)\subseteq Y\setminus \overline{f(A)}$ for some $U\in \mathcal F\setminus \{min~\!\mathcal F\}$, and hence $\emptyset \neq A\cap U\subseteq A\cap f^{-1}(Y\setminus \overline{f(A)})=\emptyset$, a contraction. Therefore, $\bigcap\limits_{\mathcal F\in \mathfrak{F}_A} K_{\mathcal F}\cap\overline{f(A)} \neq\emptyset$. Select a point $y_A\in \bigcap\limits_{\mathcal F\in \mathfrak{F}_A} K_{\mathcal F}\cap \overline{f(A)}$. Then $\overline{\{y_A\}}\subseteq \overline{f(A)}$. On the other hand, for $a\in A$, if $f(a)\not\in \overline{\{y_A\}}$, then $a\in f^{-1}(Y\setminus \overline{\{y_A\}})$. By the core compactness of $X$, there is a sequence $\mathcal F_{a}=\{U_\infty^a, ..., U_n^a, ..., U_2^a, U_1^a\}\subseteq \mathcal O(X)$ such that $a\in U_\infty^a$ and $U_\infty^a\ll ...\ll U_n^a\ll ...\ll U_2^a\ll U_1^a=f^{-1}(Y\setminus \overline{\{y_A\}})$. Then $\mathcal F_{a}\in \mathfrak{F}_A$, and whence
 $$y_A\in  K_{\mathcal F_{a}}\subseteq \ua f(f^{-1}(Y\setminus \overline{\{y_A\}}))\subseteq Y\setminus \overline{\{y_A\}}),$$
 a contradiction. Therefore, $f(A)\subseteq \overline{\{y_A\}})$. Thus $\overline{f(A)}=\overline{\{y_A\}}$, proving $A\in\wdd (X)$.
 \end{proof}

By Theorem \ref{soberequiv} and Theorem \ref{CorecomptWD}, we get the following result, which has been independently obtained by Lawson and Xi (see \cite[Theorem 3.1]{Lawson-Xi}) in a different way.

\begin{theorem}\label{corecwellf sober} Every core compact well-filtered space is sober.
\end{theorem}

Theorem \ref{corecwellf sober} gives a positive answer to Jia-Jung problem \cite{jia-2018} (see \cite[Question 2.5.19]{jia-2018}) and improves a  well-known result that every locally compact well-filtered space is sober (see, e.g., \cite{redbook, Kou}).

By Theorem \ref{SoberLC=CoreC} and Theorem \ref{corecwellf sober}, we get the following corollary.

\begin{corollary}\label{WfLC=CoreC}
	Let $X$ be a well-filtered space. Then $X$ is locally compact if{}f $X$ is core compact.
\end{corollary}

Figure 1 shows certain relations among some kinds of spaces.

\begin{figure}[ht]
	\centering
	\includegraphics[height=2.2in,width=4.0in]{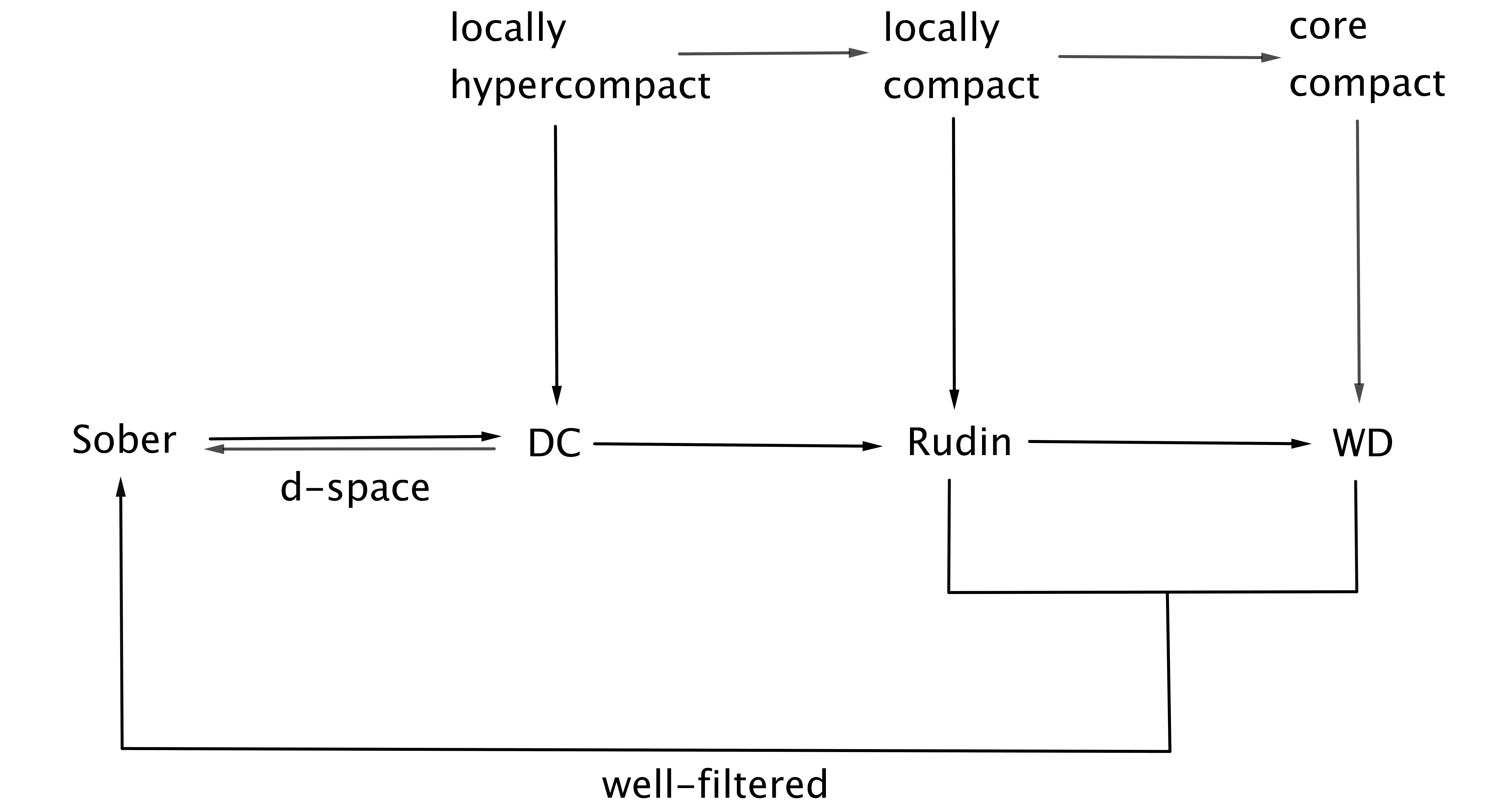}
	\caption{Certain relations among some kinds of spaces}
\end{figure}

\begin{theorem}\label{rudinWF}
	Let $X$ be a $T_0$ space. Consider the following conditions:
 \begin{enumerate}[\rm (1)]
		\item $X$ is sober.
        \item For each $(A, K)\in \ir_c(X)\times\mk (X)$, $max (A)\neq\emptyset$ and $\downarrow (A\cap K)\in \mathcal C(X)$.
        \item $X$ is well-filtered.
\end{enumerate}
Then \emph{(1)} $\Rightarrow$ \emph{(2)} $\Rightarrow$ \emph{(3)}, and all three conditions are equivalent if $X$ is core compact.
\end{theorem}

\begin{proof}  (1) $\Rightarrow$ (2):  Suppose that $X$ is sober and $(A, K)\in \ir_c(X)\times\mk (X)$. Then there is an $x\in X$ such that $A=\overline {\{x\}}$, and hence $max (A)=\{x\}\neq\emptyset$. Now we show that $\downarrow (A\cap K)=\downarrow (\downarrow x\cap K)$ is closed. If $\downarrow (\downarrow x\cap K)\neq\emptyset$ (i.e., $\downarrow x\cap K\neq\emptyset$), then $x\in K$ since $K$ is saturated (that is, $K$ is an upper set). It follows that $\downarrow (\downarrow x\cap K)=\downarrow x\in \mathcal C(X)$.

(2) $\Rightarrow$ (3): Suppose that $\mathcal K\subseteq \mathord{\mathsf K}(X)$ is filtered, $U\in \mathcal O(X)$, and $\bigcap \mathcal K \subseteq U$. If $K\not\subseteq U$ for each $K\in \mathcal K$, then by Lemma \ref{t Rudin}, $X\setminus U$ contains a minimal irreducible closed subset $A$ that still meets all members of $\mathcal{K}$. For any $\{K_1, K_2\}\subseteq \mathcal K$, we can find $K_3\in \mathcal K$ with $K_3\subseteq K_1\cap K_2$. It follows that $\downarrow (A\cap K_1)\in \mathcal C(X)$ and $\emptyset \neq A\cap K_3\subseteq \downarrow (A\cap K_1)\cap K_2\neq\emptyset$, and hence $\downarrow (A\cap K_1)=A$ by the minimality of $A$. Select an $x\in max (A)$. Then for each $K\in \mathcal K$, $x\in  \downarrow (A\cap K)$, and consequently, there is $a_k\in A\cap K$ such that $x\leq a_k$. By the maximality of $x$ we have $x=a_k$. Therefore, $x\in K$ for all $K\in \mathcal K$, and whence $x\in \bigcap \mathcal K \subseteq U\subseteq X\setminus A$, a contradiction. Thus $X$ is well-filtered.

Finally assume that $X$ is core compact and well-filtered, then by Theorem \ref{corecwellf sober}, $X$ is sober.
\end{proof}

  If $X$ is a $d$-space and $A$ a nonempty closed subset of $X$, then by Zorn's Lemma there is a maximal chain $C$ in $A$. Let $c=\vee C$. Then $c\in max (A)$. So by Theorem \ref{rudinWF} we get the following corollary.

\begin{corollary}\label{dspaceWF}
	Let $X$ be a $d$-space. Consider the following conditions:
 \begin{enumerate}[\rm (1)]
		\item $X$ is sober.
        \item For each $(A, K)\in \ir_c(X)\times\mk (X)$, $\downarrow (A\cap K)\in \mathcal C(X)$.
        \item $X$ is well-filtered.
\end{enumerate}
Then \emph{(1)} $\Rightarrow$ \emph{(2)} $\Rightarrow$ \emph{(3)}, and all three conditions are equivalent if $X$ is core compact.
\end{corollary}

\begin{corollary}\label{xi-lawsonWF}\emph{(\cite{Xi-Lawson-2017})}
	Let $X$ be a $d$-space with the property that $\downarrow (A\cap K)$ is closed whenever
$A\in \mathcal C(X)$ and $K\in \mk (X)$.  Then $X$ is well-filtered.
\end{corollary}

\begin{example}\label{examp1}
	Let $X$ be a countable infinite set and endow $X$ with the cofinite topology (having the complements of the finite sets as open sets). The
resulting space is denoted by $X_{cof}$. Then $\mk (X_{cof})=2^X\setminus \{\emptyset\}$ (that is, all nonempty subsets of $X$), and hence $X_{cof}$ is a locally compact and first countable $T_1$ space. By Theorem \ref{LCrudin}, $X_{cof}$ is a Rudin space (and hence a $\wdd$-space). Let $\mathcal K=\{X\setminus F : F\in X^{(<\omega)}\}$. It is easy to check that $\mathcal K\subseteq \mk (X_{cof})$ is filtered and $X\in m(\mathcal K)$. Therefore, $X\in \kf(X)$ but $X\not\in \md_c(X)$.  Thus $\kf(X)\neq \md_c(X)$. $X_{cof}$ is not sober, and hence $X_{cof}$ is not well-filtered by Theorem \ref{corecwellf sober}.
\end{example}

\begin{example}\label{examp2}
	Let $L$ be the complete lattice constructed by Isbell \cite{isbell}. Then by \cite[Corollary 3.2]{Xi-Lawson-2017}, $\Sigma L$ is a well-filtered space. Note that it is not sober. Then by Theorem \ref{soberequiv}, it is not a $\wdd$ space (hence not a Rudin space).  So $\wdd(X)\neq\ir_c(X)$ and $\kf(X)\neq\ir_c(X)$.
\end{example}

\begin{lemma}\label{WDimage}
Let $X,Y$ be two $T_0$ spaces. If $f:X\longrightarrow Y$ is a continuous mapping and $A\in \wdd (X)$, then $\overline{f(A)}\in \wdd (Y)$.
\end{lemma}
\begin{proof}	Let $Z$ is a well-filtered space and $g:Y\longrightarrow Z$ is a continuous mapping.
Since $g\circ f:X\longrightarrow Z$ is continuous and $A\in \wdd (X)$, there is $z\in Z$ such that $\overline{g(\overline{f(A)})}=\overline{g\circ f(A)}=\overline{\{z\}}$. Thus $\overline{f(A)}\in \wdd (Y)$.
\end{proof}

\begin{proposition}\label{WDretract}
A retract of a well-filtered determined space is well-filtered determined.
\end{proposition}
\begin{proof}
	Assume $X$ is a well-filtered determined space and $Y$ a retract of $X$.  Then there are continuous mappings $f:X\longrightarrow Y$ and $g:Y\longrightarrow X$ such that $f\circ g=id_Y$.  Let $B\in \ir_c(Y)$. Then $\overline{g(B)}\in\ir_c(X)$ by Lemma \ref{irrsubspace} and Lemma \ref{irrimage}. As $X$ is well-filtered determined, $\overline{g(B)}\in \wdd (X)$). By Lemma \ref{WDimage}, $B=\overline{f(g(B))}=\overline{f(\overline{g(B)})}\in \wdd (Y)$. Hence, $Y$ is well-filtered determined.
\end{proof}

\begin{definition}\label{Xinfty} For a $T_0$ space $X$, select a point $\infty$ such that $\infty\not\in X$. Then $\mathcal C(X)\cup \{X\cup \{\infty\}\}$ (as the set of all closed sets) is a topology on $X\cup\{\infty\}$. The resulting space is denoted by $X_{\infty}$.
\end{definition}

\begin{lemma}\label{lem5}
	If $X$ is a well-filtered space, then $X_{\infty}$ is a well-filtered space.
\end{lemma}\label{WFinfty}
\begin{proof}
	We first show that $X_{\infty}$ is $T_0$. Let $x,y\in X_{\infty}$ with $x\neq y$. There are two cases:
	
	{Case 1:} $x,y\in X$. Then we have $\cl_{X_\infty}\{x\}=\cl_X\{x\}\neq\cl_X\{y\}=\cl_{X_\infty}\{y\}$.
	
	{Case 2:}  $x\in X$ and $y=\infty$. Note that $\cl_{X_\infty}\{\infty\}=X_{\infty}$ and $\cl_{X_\infty}\{x\}\subseteq X$. It follows that $\cl_{X_\infty}\{\infty\}\neq\cl_{X_\infty}\{x\}$.
	
	Thus $X_{\infty}$ is $T_0$. Let $\{K_i:i\in I\}\subseteq\mk(X_{\infty})$ be a filtered family and $U\in \mathcal O(X_{\infty})$ such that $\bigcap_{i\in I}K_i\subseteq U$. Note that $\infty$ is the largest element in $X$ with respect to the specialization order, so
	$\infty\in \bigcap_{i\in I}K_i\subseteq U$. Let $V=U\setminus\{\infty\}=X\setminus(X_{\infty}\setminus U)$. Then $V\in\mathcal O(X)$ and $U=V\cup\{\infty\}$. For each $i\in I$, let $K^*_i=K_i\setminus\{\infty\}$. One can easily check that $\{K^*_i:i\in I\}\subseteq \mk(X)$ is a filtered family and $\bigcap_{i\in I}K^*_i\subseteq V$. Since $X$ is well-filtered, there exists
	$i_0\in I$ such that $K^*_{i_0}\subseteq V$, which implies that $K_{i_0}\subseteq U$. Thus $X_\infty$ is well-filtered.
\end{proof}

\begin{proposition}\label{WDclosed}
Every closed subspace of a well-filtered determined space is well-filtered determined.
\end{proposition}
\begin{proof}
	Let $X$ be a well-filtered determined space and $A\in\mathcal C(X)$. We need to show $A$, as a subspace of $X$, is well-filtered determined.
	Let $B\in\ir_c(A)$ and $f: A\longrightarrow Y$ a continuous mapping to a well-filtered space $Y$. Then by Lemma \ref{WFinfty}, $Y_{\infty}$ is well-filtered. Define a mapping $f_{\infty}:X\longrightarrow Y_{\infty}$ as follows:
	$$f_{\infty}(x)=
	\begin{cases}
	f(x)& x\in A\\
	\infty& x\notin A.
	\end{cases}$$
Then $f_{\infty}$ is continuous since for each $C\in\mathcal C(Y_{\infty})$, it holds that

	$$f_{\infty}^{-1}(C)=
\begin{cases}
f^{-1}(C)& \infty\notin C\\
X& \infty\in C.
\end{cases}$$

Since $X$ is well-filtered determined, there exists $y_B\in Y_{\infty}$ such that $\overline{f_{\infty}(B)}=\overline{f(B)}=\overline{\{y_B\}}$. Clearly, $y_B\in Y$. So $B\in \wdd (A)$. Thus $A$ is well-filtered determined.

\end{proof}

\begin{lemma}\label{WDsetprod}
	Let	$\{X_i: 1\leq i\leq n\}$ be a finite family of $T_0$ spaces and $X=\prod\limits_{i=1}^{n}X_i$ the product space. For $A\in\ir (X)$, the following conditions are equivalent:
\begin{enumerate}[\rm (1)]
	\item $A$ is a $\wdd$ set.
	\item $p_i(A)$ is a $\wdd$ set for each $1\leq i\leq n$.
\end{enumerate}
\end{lemma}

\begin{proof} (1) $\Rightarrow$ (2): By Lemma \ref{WDimage}.

(2) $\Rightarrow$ (1): By induction, we need only to prove the implication for the case of $n=2$. Let $A_1=\cl_{X_1} p_1(A)$ and $A_2=\cl_{X_2} p_2(A)$. Then by condition (2), $(A_1, A_2)\in \mathcal \wdd (X_1)\times \wdd (X_2)$. Now we show that the product $A_1\times A_2\in\wdd (X)$. Let $f : X_1\times X_2 \longrightarrow Y$ a continuous mapping from $X_1\times X_2$ to a well-filtered space $Y$. For each $b\in X_2$, $X_1$ is homeomorphic to $X_1\times \{b\}$ (as a subspace of $X_1\times X_2$) via the homeomorphism $\mu_b : X_1 \longrightarrow X_1\times \{b\}$ defined by $\mu_b(x)=(x, b)$. Let $i_b : X_1\times \{b\}\longrightarrow X_1\times X_2$ be the embedding of $X_1\times \{b\}$ in $X_1\times X_2$. Then  $f_{b}=f\circ i_b \circ \mu_b : X_1 \longrightarrow Y$, $f_b(x)=f((x, b))$,  is continuous. Since $A_1\in \mathcal \wdd (X_1)$, there is a unique $y_b\in Y$ such that $\overline{f(A_1\times \{b\})}=\overline{f_b(A_1)}=\overline{\{y_b\}}$. Define a mapping $g_A : X_2 \longrightarrow Y$ by $g_A(b)=y_b$. For each $V\in \mathcal O(Y)$,
$$\begin{array}{lll}
	g_A^{-1}(V)& =\{b\in X_2 : g_A(b)\in V\}\\
	           & =\{b\in X_2 : \overline{f_b(A_1)}\cap V\neq\emptyset\}\\
	           & =\{b\in X_2 : \overline{f(A_1\times \{b\})}\cap V\neq\emptyset\}\\
	           & =\{b\in X_2 : f(A_1\times \{b\})\cap V\neq\emptyset\}\\
               & =\{b\in X_2 : (A_1\times \{b\})\cap f^{-1}(V)\neq\emptyset\}.\\
	\end{array}$$
 Therefore, for each $b\in g_A^{-1}(V)$, there is an $a_1\in A_1$ such that $(a_1, b)\in f^{-1}(V)\in \mathcal O(X_1\times X_2)$, and hence there is $(U_1, U_2)\in \mathcal O(X_1)\times \mathcal O(X_2)$ such that $(a_1, b)\in U_1\times U_2\subseteq  f^{-1}(V)$. It follows that $b\in U_2\subseteq g_A^{-1}(V)$. Thus $g_A : X_2 \longrightarrow Y$ is continuous. Since $A_2\in \mathcal \wdd (X_1)$, there is a unique $y_A\in Y$ such that $\overline{g_A(A_2)}=\overline{\{y_A\}}$. Therefore, by Lemma \ref{irrprod}, we have
 $$\begin{array}{lll}
     \overline{f(\cl_X A)} & =\overline{f(A_1\times A_2)}\\
	                       & =\overline{\bigcup\limits_{a_2\in A_2}f(A_1\times \{a_2\})}\\
	                       & =\overline{\bigcup\limits_{a_2\in A_2}\overline{f(A_1\times \{a_2\})}}\\
	                       & =\overline{\bigcup\limits_{a_2\in A_2}\overline{\{g_A(a_2)\}}}\\
                           & =\overline{\bigcup\limits_{a_2\in A_2}\{g_A(a_2)\}}\\
                           & =\overline{g_A(A_2)}\\
                           & =\overline{\{y_A\}}.\\
	\end{array}$$
Thus $\cl_X A\in \wdd (X)$, and hence $A$ is a $\wdd$ set.
\end{proof}

By Corollary \ref{irrcprod} and Lemma \ref{WDsetprod}, we get the following result.

\begin{corollary}\label{WDclosedsetprod}
	Let	$X=\prod\limits_{i=1}^{n}X_i$ be the product of a finitely family $\{X_i: 1\leq i\leq n\}$ of $T_0$ spaces. If $A\in\wdd (X)$, then $A=\prod\limits_{i=1}^{n}p_i(X_i)$, and $p_i(A)\in \wdd (X_i)$ for all $1\leq i \leq n$.
\end{corollary}

\begin{theorem}\label{rudinprod}
	Let $\{X_i: 1\leq i\leq n\}$ be a finitely family of $T_0$ spaces. Then the following two conditions are equivalent:
	\begin{enumerate}[\rm(1)]
		\item The product space $\prod\limits_{i=1}^{n}X_i$ is a well-filtered determined space.
		\item For each $1\leq i \leq n$, $X_i$ is a well-filtered determined space.
	\end{enumerate}
\end{theorem}

\begin{proof}	
	(1) $\Rightarrow$ (2):  For each $1\leq i \leq n$, $X_i$ is a retract of $\prod\limits_{i=1}^{n}X_i$ . By Proposition \ref{WDretract}, $X_i$ is a well-filtered determined space.
	
	(2) $\Rightarrow$ (1): Let $X=\prod\limits_{i=1}^{n} X_i$. For any $A\in \ir_c(X)$, by Corollary \ref{irrcprod} and Lemma \ref{WDsetprod}, we have $A\in \wdd(X)$, proving that $X$ is a well-filtered determined space.
\end{proof}

\section{A direct construction of well-filtered reflections of $T_0$ spaces}

Section 7 is devoted to the reflection of category of well-filtered spaces in that of $T_0$ spaces. Using $\wdd$ sets, we present a direct construction of the  well-filtered reflections of $T_0$ spaces, and show that the product of any family of well-filtered spaces is well-filtered. Some important properties of well-filtered reflections of $T_0$ spaces are investigated.

\begin{definition}\label{WFtion}
	Let $X$ be a $T_0$ space. A \emph{well-filtered reflection} of $X$ is a pair $\langle \widetilde{X}, \mu\rangle$ consisting of a well-filtered space $\widetilde{X}$ and a continuous mapping $\mu :X\longrightarrow \widetilde{X}$ satisfying that for any continuous mapping $f: X\longrightarrow Y$ to a well-filtered space, there exists a unique continuous mapping $f^* : \widetilde{X}\longrightarrow Y$ such that $f^*\circ\mu=f$, that is, the following diagram commutes.\\
\begin{equation*}
	\xymatrix{
		X \ar[dr]_-{f} \ar[r]^-{\mu}
		&\widetilde{X}\ar@{.>}[d]^-{f^*}\\
		&Y}
	\end{equation*}

\end{definition}

Well-filtered reflections, if they exist, are unique up to homeomorphism. We shall use $X^w$ to denote the space of the well-filtered reflection of $X$ if it exists.

	Let $X$ be a $T_0$ space. Then by Proposition \ref{DRWIsetrelation}, $\wdd (X)\subseteq \ir_c(X)$, and whence the space $P_H(\wdd(X))$ has the topology $\{\Diamond U : U\in \mathcal O(X)\}$, where
$\Diamond U=\{A\in \wdd(X) : A\cap U\neq\emptyset\}$. The closed subsets of $P_H(\wdd(X))$ are exactly the set of forms $\Square C=\downarrow_{\wdd(X)}C$ with $C\in\mathcal C(X)$.

\begin{lemma}\label{lemmaclosure}
	Let $X$ be a $T_0$ space and $A\subseteq X$. Then $\overline{\eta_X(A)}=\overline{\eta_X\left(\overline{A}\right)}=\overline{\Box A}=\Box \overline{A}$ in $P_H(\wdd(X))$.
\end{lemma}
\begin{proof}
	Clearly, $\eta_X(A)\subseteq \Box A\subseteq \Box\overline{A}$, $\eta_X\left(\overline{A}\right)\subseteq \Box\overline{A}$ and $\Box\overline{A}$ is closed in $P_H(\wdd(X))$. It follows that
	$$\overline{\eta_X(A)}\subseteq \overline{\Box A}\subseteq \Box \overline{A}\ \text{ and }\ \overline{\eta_X(A)}\subseteq\overline{\eta_X\left(\overline{A}\right)}\subseteq\Box\overline{A}.$$
	To complete the proof, we need to show $\Box\overline{A}\subseteq \overline{\eta_X(A)}$.
	Let $F\in \Box\overline{A}$. Suppose $U\in\mathcal O(X)$ such that $F\in\Diamond U$, that is, $F\cap U\neq\emptyset$. Since $F\subseteq \overline{A}$, we have $A\cap U\neq\emptyset$. Let $a\in A\cap U$. Then $\da a\in \Diamond U\cap \eta_X(A)\neq\emptyset$. This implies that $F\in \overline{\eta_X(A)}$. Whence $\Box\overline{A}\subseteq \overline{\eta_X(A)}$.	
\end{proof}

\begin{lemma}\label{lemmaeta}
	The mapping $\eta_X:X\longrightarrow P_H(\wdd(X))$ defined by
	$$\forall x\in X, \ \eta_X(x)=\da x,$$
	is a topological embedding.
\end{lemma}
\begin{proof}
For $U\in\mathcal O(X)$, we have $$\eta_X^{-1}(\Diamond U)=\{x\in X: \da x\in\Diamond U\}=\{x\in X: x\in U\}=U,$$ so $\eta_X$ is continuous.
In addition, we have $$
\eta_X(U)=\{\da x: x\in U\}
=\{\da x: \da x\in\Diamond U\}
=\Diamond U\cap \eta_X(X),$$
which implies that $\eta_X$ is an open mapping to $\eta_X(X)$, as a subspace of $P_H(\wdd(X))$.
As $\eta_X$ is an injection, $\eta_X$ is a topological embedding.
\end{proof}

\begin{lemma}\label{lemmaWDirr}
	Let $X$ be a $T_0$ space and $A$ a nonempty subset of $X$. Then the following conditions are equivalent:
	\begin{enumerate}[\rm (1)]
		\item $A$ is irreducible in $X$.
		\item $\Box A$ is irreducible in $P_H(\wdd(X))$.
        \item $\Box \overline{A}$ is irreducible in $P_H(\wdd(X))$.
	\end{enumerate}
\end{lemma}
\begin{proof}
	(1) $\Rightarrow$ (3): Assume $A$ is irreducible. Then $\eta_X(A)$ is irreducible in $P_H(\wdd(X))$ by Lemma \ref{irrimage} and Lemma \ref{lemmaeta}. By Lemma \ref{irrsubspace} and Lemma \ref{lemmaclosure}, $\Box \overline{A}=\overline{\eta_X(A)}$ is irreducible in $P_H(\wdd(X))$.
	
	(3) $\Rightarrow$ (1): Assume $\Box \overline{A}$ is irreducible. Let $A\subseteq B\cup C$ with  $B,C\in\mathcal C(X)$. By Proposition \ref{DRWIsetrelation}, $\wdd (X)\subseteq \ir_c(X)$, and consequently, we have $\Box\overline{A}\subseteq \Box B\cup\Box C$. Since $\Box\overline{A}$ is irreducible,  $\Box\overline{A}\subseteq \Box B$ or $\Box\overline{A}\subseteq C$, showing that  $\overline{A}\subseteq B$ or $\overline{A}\subseteq C$, and consequently, $A\subseteq B$ or $A\subseteq C$, proving $A$ is irreducible.

    (2) $\Leftrightarrow$ (3): By By Lemma \ref{irrsubspace} and Lemma \ref{lemmaclosure}.

\end{proof}

\begin{lemma}\label{lemmafstar}
Let $X$ be a $T_0$ space and $f:X\longrightarrow Y$ a continuous mapping from $X$ to a well-filtered space $Y$. Then there exists a unique continuous mapping $f^* :P_H(\wdd(X))\longrightarrow Y$ such that $f^*\circ\eta_X=f$, that is, the following diagram commutes.
\begin{equation*}
\xymatrix{
	X \ar[dr]_-{f} \ar[r]^-{\eta_X}
	&P_H(\wdd(X))\ar@{.>}[d]^-{f^*}\\
	&Y}
\end{equation*}	
\end{lemma}
\begin{proof}For each $A\in\wdd(X)$, there exists a unique $y_A\in Y$ such that $\overline{f(A)}=\overline{\{y_A\}}$. Then we can define a mapping $f^*:P_H(\wdd(X))\longrightarrow Y$ by
$$\forall A\in\wdd(X),\ \ f^*(A)=y_A.$$

{Claim 1:}  $f^*\circ \eta_X=f$.

Let $x\in X$. Since $f$ is continuous, we have
$\overline{f\left(\overline{\{x\}}\right)}=\overline{f(\{x\})}=\overline{\{f(x)\}}$,
so
$f^*\left(\overline{\{x\}}\right)=f(x)$. Thus $f^*\circ \eta_X=f$.

{Claim 2:}  $f^*$ is continuous.

Let $V\in\mathcal O(Y)$. Then
$$\begin{array}{lll}
(f^*)^{-1}(V)&=&\{A\in\wdd(X): f^*(A)\in V\}\\
&=&\{A\in\wdd(X): \overline{\{f^*(A)\}}\cap V\neq\emptyset\}\\
&=&\{A\in\wdd(X): \overline{f(A)}\cap V\neq\emptyset\}\\
&=&\{A\in\wdd(X): f(A)\cap V\neq\emptyset\}\\
&=&\{A\in\wdd(X): A\cap f^{-1}(V)\neq\emptyset\}\\
&=&\Diamond f^{-1}(V),
\end{array}$$
which shows that $(f^*)^{-1}(V)$ is open in $P_H(\wdd(X))$. Thus  $f^*$ is continuous.

{Claim 3:}  The mapping $f^*$ is unique such that $f^*\circ \eta_X=f$.

Assume $g:P_H(\wdd(X))\longrightarrow Y$ is a continuous mapping such that $g\circ\eta_X=f$.  Let $A\in\wdd(X)$. We need to show $g(A)=f^*(A)$.
Let $a\in A$.  Then $\overline{\{a\}}\subseteq A$, implying that $g(\overline{\{a\}})\leq_Y g(A)$, that is,  $g\left(\overline{\{a\}}\right)=f(a)\in\overline{\{g(A)\}}$. Thus $\overline{\{f^*(A)\}}=\overline{f(A)}\subseteq \overline{\{g(A)\}}$.
In addition, since $A\in\overline{\eta_X(A)}$ and $g$ is continuous, $g(A)\in g\left(\overline{\eta_X(A)}\right)\subseteq\overline{g(\eta_X(A))}=\overline{f(A)}=\overline{\{f^*(A)\}}$, which implies that $\overline{\{g(A)\}}\subseteq \overline{\{f^*(A)\}}$.  So $\overline{\{g(A)\}}=\overline{\{f^*(A)\}}$. Since $Y$ is $T_0$, $g(A)=f^*(A)$. Thus $g=f^*$.
\end{proof}

\begin{lemma}\label{lemmaBoxwd}
	Let $X$ be a $T_0$ space and $C\in\mathcal C(X)$. Then the following conditions are equivalent:
	\begin{enumerate}[\rm (1)]
		\item $C$ is well-filtered determined in $X$.
		\item $\Box C$ is well-filtered determined in $P_H(\wdd(X))$.
	\end{enumerate}
\end{lemma}
\begin{proof}
(1) $\Rightarrow$ (2): By Propositions \ref{WDimage}, Lemma \ref{lemmaclosure} and Lemma \ref{lemmaeta}.

(2) $\Rightarrow$ (1).  Let $Y$ be a well-filtered space and $f:X\longrightarrow Y$  a continuous mapping. By Lemma \ref{lemmafstar}, there exists a continuous mapping $f^* :P_H(\wdd(X))\longrightarrow Y$ such that $f^*\circ\eta_X=f$.
Since $\Box C=\overline{\eta_X(C)}$ is well-filtered determined and $f^*$ is continuous,  there exists a unique $y_C\in Y$ such that
$\overline{f^*\left(\overline{\eta_X(C)}\right)}=\overline{\{y_C\}}$. Furthermore, we have
$$\overline{\{y_C\}}=\overline{f^*\left(\overline{\eta_X(C)}\right)}=\overline{f^*(\eta_X(C))}=\overline{f(C)}.$$
So $C$ is well-filtered determined.
\end{proof}

\begin{theorem}\label{WDwf}
	Let $X$ be a $T_0$ space. Then $P_H(\wdd(X))$ is a well-filtered space.
\end{theorem}
\begin{proof}
	Since $X$ is $T_0$, one can deduce that $P_H(\wdd(X))$ is $T_0$. Let $\{
	\mathcal K_i:i\in I\}\subseteq\mk(P_H(\wdd(X)))$ be a filtered family and $U\in\mathcal O(X)$ such that
	$\bigcap_{i\in I}\mathcal K_i\subseteq \Diamond U$. We need to show $\mathcal K_i\subseteq\Diamond U$ for some $i\in I$.
	Assume, on the contrary, $\mathcal K_i\nsubseteq \Diamond U$, i.e., $\mathcal K_i\cap\Box(X\setminus U)\neq\emptyset$,  for any $i\in I$.
	
	Let $\mathcal A=\{C\in\mathcal C(X): C\subseteq X\setminus U \text{ and } \mathcal K_i\cap \Box C\neq\emptyset \text{ for all }  i\in I\}$. Then we have the following two facts.
	
	{(a1)} $\mathcal A\neq\emptyset$ because $X\setminus U\in\mathcal A$.
	
	{(a2)} For any filtered family $\mathcal F\subseteq\mathcal A$, $\bigcap\mathcal F\in\mathcal A$.
	
	Let $F=\bigcap\mathcal F$. Then $F\in \mathcal C(X)$ and $F\subseteq X\setminus U$. Assume, on the contrary, $F\notin\mathcal A$. Then there exists $i_0\in I$ such that $\mathcal K_{i_0}\cap \Box F=\emptyset$. Note that $\Box F=\bigcap_{C\in\mathcal F}\Box C$, implying that $\mathcal K_{i_0}\subseteq\bigcup_{C\in\mathcal F}\Diamond (X\setminus C)$ and $\{\Diamond (X\setminus C) : C\in\mathcal F\}$ is a directed family since $\mathcal F$ is filtered. Then there is $C_0\in\mathcal F$ such that $\mathcal K_{i_0}\subseteq \Diamond (X\setminus C_0)$, i.e., $\mathcal K_{I_0}\cap\Box C_0=\emptyset$,  contradicting $C_0\in\mathcal A$. Hence $F\in\mathcal A$.
	
	By Zorn's Lemma, there exists a minimal element $C_m$ in $\mathcal A$ such that $\Box C_m$ intersects all members of $\mathcal K$. Clearly,  $\Box C_m$ is also a minimal closure set that  intersects all members of $\mathcal K$, hence is a Rudin set in $P_H(\wdd(X))$. By Proposition \ref{DRWIsetrelation} and Lemma \ref{lemmaBoxwd}, $C_m$ is well-filtered determined. So $C_m\in \Box C_m\cap\bigcap\mathcal K\neq\emptyset$. It follows that $\bigcap\mathcal K\nsubseteq \Diamond(X\setminus C_m)\supseteq \Diamond U$, which implies that $\bigcap\mathcal K\nsubseteq \Diamond U$, a contradiction.
\end{proof}

By Lemma \ref{lemmafstar} and Theorem \ref{WDwf}, we have the following result.

\begin{theorem}\label{WFilterification}
	Let $X$ be a $T_0$ space and $X^w=P_H(\wdd(X))$. Then the pair $\langle X^w, \eta_X\rangle$, where $\eta_X :X\longrightarrow X^w$, $x\mapsto\overline{\{x\}}$, is the well-filtered reflection of $X$.
\end{theorem}

\begin{corollary}\label{WFreflective}
	The category $\mathbf{Top}_w$ of all well-filtered spaces is a reflective full subcategory of  $\mathbf{Top}_0$.
\end{corollary}

\begin{corollary}\label{WFfuctor}
	Let $X,Y$ be two $T_0$ spaces and $f:X\longrightarrow Y$  a continuous mapping. Then there exists a unique continuous mapping $f^w:X^w\longrightarrow Y^w$ such that $f^w\circ \eta_X=\eta_Y\circ f$, that is, the following diagram commutes.
		\begin{equation*}
	\xymatrix{
		X \ar[d]_-{f} \ar[r]^-{\eta_X} &X^w\ar[d]^-{f^w}\\
		Y \ar[r]^-{\eta_Y} &Y^w
	}
	\end{equation*}
For each $A\in \wdd (X)$, $f^w(A)=\overline{f(A)}$.
\end{corollary}

Corollary \ref{WFfuctor} defines a functor $W : \mathbf{Top}_0 \longrightarrow \mathbf{Top}_w$, which is the left adjoint to the inclusion functor $I : \mathbf{Top}_w \longrightarrow \mathbf{Top}_0$.

\begin{corollary}\label{WFwdc}
	For a $T_0$ space $X$, the following conditions are equivalent:
	\begin{enumerate}[\rm (1)]
		\item $X$ is well-filtered.
		\item $\mathsf{RD}(X)=\mathcal S_c(X)$.
        \item $\wdd (X)=\mathcal S_c(X)$, that is, for each $A\in\wdd(X)$, there exists a unique $x\in X$ such that $A=\overline{\{x\}}$.
        \item $X\cong X^w$.

	\end{enumerate}
\end{corollary}
\begin{proof}  (1) $\Rightarrow$ (2): Applying Lemma \ref{rudinwf} to the identity $id_X : X \longrightarrow X$.

(2) $\Rightarrow$ (3): By Proposition \ref{DRWIsetrelation}.

(3) $\Rightarrow$ (4): By assumption, $\wdd(X)=\left\{\overline{\{x\}}:x\in X\right\}$, so $X^w=P_H(\wdd(X))=P_H(\{\overline {\{x\}} : x\in X\})$, and whence $X\cong X^w$.

(4) $\Rightarrow$ (1): By Theorem \ref{WDwf} or by Proposition \ref{DRWIsetrelation} and Corollary \ref{WFwdc}.
\end{proof}

The equivalence of (1) and (2) in Corollary \ref{WFwdc} has been proved in \cite{Shenchon} in a different way.

By Proposition \ref{d-spacecharac1}, Proposition \ref{DRWIsetrelation} and Corollary \ref{WFwdc}, we get the following known result (see, e.g., \cite[Proposition 2.1]{Xi-Lawson-2017})

\begin{corollary}\label{wf-dspace} A well-filtered space is a $d$-space.
\end{corollary}

\begin{corollary}\label{wfretract}\emph{(\cite{xi-zhao-MSCS-well-filtered})}  A retract of a well-filtered space is well-filtered.
\end{corollary}
\begin{proof} Suppose that $Y$ is a retract of a well-filtered space $X$. Then there are continuous mappings $f : X\longrightarrow Y$ and $g : Y\longrightarrow X$ such that $f\circ g=id_Y$. Let $B\in \mathsf{RD}(Y)$, then by Lemma \ref{WDimage} and Corollary \ref{WFwdc}, there exists a unique $x_B\in X$ such that $\overline{g(B)}=\overline{\{x_B\}}$. Therefore, $\overline{B}=\overline{f\circ g(B)}=\overline{f(\overline{g(B)})}=\overline{f(\overline{\{x_B\}})}=\overline{\{f(x_B)\}}$. By Corollary \ref{WFwdc}, $Y$ is well-filtered.
\end{proof}

\begin{theorem}\label{wfreflectionprod}
	Let $\{X_i: 1\leq i\leq n\}$ be a finitely family of $T_0$ spaces. Then $(\prod\limits_{i=1}^{n}X_i)^w=\prod\limits_{i=1}^{n}X_i^w$ (up to homeomorphism).
\end{theorem}

\begin{proof}	
	Let $X=\prod\limits_{i=1}^{n}X_i$. By Corollary \ref{WDclosedsetprod}, we can define a mapping $\gamma : P_H(\wdd (X))  \longrightarrow \prod\limits_{i=1}^{n}P_H(\wdd (X_i))$ by

\begin{center}
$\forall A\in \wdd (X)$, $\gamma (A)=(p_1(A), p_2(A), ..., p_n(A))$.
\end{center}

By Lemma \ref{WDsetprod} and Corollary \ref{WDclosedsetprod}, $\gamma$ is bijective. Now we show that $\gamma$ is a homeomorphism. For any $(U_1, U_2, ..., U_n)\in \mathcal O(X_1)\times \mathcal O(X_2)\times ... \times \mathcal O(X_n)$, by Lemma \ref{WDsetprod} and Corollary \ref{WDclosedsetprod}, we have

$$\begin{array}{lll}
\gamma^{-1}(\Box U_1\times \Box U_2\times ... \times\Box U_n)&=&\{A\in\wdd(X): \gamma(A)\in \Box U_1\times \Box U_2\times ... \times\Box U_n\}\\
&=&\{A\in\wdd(X): p_1(A)\subseteq U_1, p_2(A)\subseteq U_2, ..., p_n(A)\subseteq U_n\}\\
&=&\{A\in\wdd(X): A\subseteq U_1\times U_2\times ... \times U_n\}\\
&=&\Box U_1\times U_2\times ... \times U_n\in \mathcal O(P_H(\wdd (X)), \mbox{~and~}
\end{array}$$

$$\begin{array}{lll}
\gamma (\Box U_1\times U_2\times ... \times U_n)&=&\{\gamma (A): A\in \wdd (X) \mbox{~and~} A\subseteq U_1\times U_2\times ... \times U_n \}\\
&=&\Box U_1\times \Box U_2\times ... \times\Box U_n\in O(\prod\limits_{i=1}^{n}P_H(\wdd (X_i))).
\end{array}$$

Therefore, $\gamma : P_H(\wdd (X))  \longrightarrow \prod\limits_{i=1}^{n}P_H(\wdd (X_i))$ is a homeomorphism, and hence $X^w$ ($=P_H(\wdd (X)$) and $\prod\limits_{i=1}^{n}X_i^w$ ($=\prod\limits_{i=1}^{n}P_H(\wdd (X_i))$ are homeomorphic.
\end{proof}

Using $\wdd$ sets and Corollary \ref{WFwdc}, we can present a simple proof the following result, which is proved in \cite{Shenchon} by using Rudin sets.

\begin{theorem}\label{WFprod} \emph{(\cite{Shenchon})}
	Let $\{X_i:i\in I\}$ be a family of $T_0$ spaces. Then the following two conditions are equivalent:
	\begin{enumerate}[\rm(1)]
		\item The product space $\prod_{i\in I}X_i$ is well-filtered.
		\item For each $i \in I$, $X_i$ is well-filtered.
	\end{enumerate}
\end{theorem}
\begin{proof}	
	(1) $\Rightarrow$ (2):  For each $i \in I$, $X_i$ is a retract of $\prod_{i\in I}X_i$. By Corollary \ref{wfretract}, $X_i$ is well-filtered.
	
	(2) $\Rightarrow$ (1): Let $X=\prod_{i\in I}X_i$. Suppose $A\in \wdd (X)$. Then by Proposition \ref{DRWIsetrelation} and Lemma \ref{WDimage}, $A\in \ir_c(X)$ and for each $i \in I$, $\cl_{X_i}(p_i(A))\in \wdd (X_i)$, and consequently, there is a $u_i\in X_i$ such that $\cl_{X_i}(p_i(A))=cl_{X_i}\{u_i\}$ by condition (2) and Corollary \ref{WFwdc}. Let $u=(u_i)_{i\in I}$. Then by Lemma \ref{irrprod} and \cite[Proposition 2.3.3]{Engelking}), we have $A=\prod_{i\in I}\cl_{X_i}(p_i(A))=\prod_{i\in I}\cl_{u_i}\{u_i\}=\cl_X \{u\}$. Whence $X$ is well-filtered by Corollary \ref{WFwdc}.
\end{proof}

\begin{theorem}\label{Wfication=sober}
	For a $T_0$ space $X$, the following conditions are equivalent:
	\begin{enumerate}[\rm (1)]
		\item $X^w$ is the sobrification of $X$, in other words, the well-filtered reflection of $X$ and sobrification of $X$ are the same.
        \item $X^w$ is sober.
		\item $X$ is well-filtered determined, that is, $\wdd(X)=\ir_c(X)$.
	\end{enumerate}
\end{theorem}

\begin{proof}
(1) $\Rightarrow$ (2): Trivial.

(2) $\Rightarrow$ (3): Let $\eta_X^w : X\longrightarrow X^w$ be the canonical topological embedding defined by $\eta_X^w(x)=\overline{\{x\}}$ (see Theorem \ref{WFilterification}). Since the pair $\langle X^s, \eta_{X}^s\rangle$, where $\eta_{X}^s :X\longrightarrow X^s=P_H(\ir_c(X))$, $x\mapsto\overline{\{x\}}$, is the soberification of $X$ and $X^w$ is sober, there exists a unique continuous mapping $\eta_{X}^{w*} :X^s\longrightarrow X^w$ such that $\eta_{X}^{w*}\circ\eta_X^s=\eta_X^w$, that is, the following diagram commutes.
\begin{equation*}
\xymatrix{
	X \ar[dr]_-{\eta_X^w} \ar[r]^-{\eta_X^s}
	&X^s\ar@{.>}[d]^-{\eta_{X}^{w*}}\\
	&X^w}
\end{equation*}	
So for each $A\in\ir_c(X)$, there exists a unique $B\in \wdd (X)$ such that $\downarrow_{\wdd (X)} A=\overline{\eta_X^w(A)}=\overline{\{B\}}=\downarrow_{\wdd (X)}B $. Clearly, we have $B\subseteq A$. On the other hand, for each $a\in A, \overline {\{a\}}\in \downarrow_{\wdd (X)} A=\downarrow_{\wdd (X)}B$, and whence $\overline {\{a\}}\subseteq B$. Thus $A\subseteq B$, and consequently, $A=B$. Thus $A\in \wdd(X)$.

(3) $\Rightarrow$ (1): If $\wdd(X)=\ir_c(X)$, then $X^w=P_H(\wdd (X))=P_H(\ir_c(X))=X^s$, with $\eta_X^w=\eta_X^s : X\longrightarrow X^w$, is the sobrification of $X$.

\end{proof}

\begin{proposition}\label{WFicationcomp}
	A $T_0$ space $X$ is compact if{}f $X^w$ is compact.
\end{proposition}
\begin{proof} By Proposition \ref{DRWIsetrelation}, we have $\mathcal S_c(X)\subseteq \wdd (X)\subseteq \ir_c (X)$. Suppose that $X$ is compact. For $\{U_i : i\in I\}\subseteq \mathcal O(X)$, if $\wdd (X)\subseteq \bigcup_{i\in I}\Diamond U_i$, then $X\subseteq \bigcup_{i\in I} U_i$ since $\mathcal S_c(X)\subseteq \wdd (X)$, and consequently, $X\subseteq \bigcup_{i\in I_0} U_i$ for some $I_0\in I^{(<\omega)}$. It follows that $\wdd (X)\subseteq \bigcup_{i\in I_0} \Diamond U_i$. Thus $X^w$ is compact by Alexander's Subbase Lemma (see, eg., \cite[Proposition I-3.22]{redbook}). Conversely, if $X^w$ is compact and $\{V_j : j\in J\}$ is a open cover of $X$, then $\wdd (X)\subseteq \bigcup_{j\in J}\Diamond V_j$. By the compactness of $X^w$, there is a finite subset $J_0\subseteq J$ such that $\wdd (X)\subseteq \bigcup_{j\in J_0}\Diamond V_j$, and whence $X\subseteq \bigcup_{j\in J_0}V_j$, proving the compactness of $X$.
\end{proof}

Since $\mathcal{S}_c(X)\subseteq \mathsf{WD}(X)\subseteq\ir_c(X)$ (see Proposition \ref{DRWIsetrelation}), the correspondence $U \leftrightarrow \Diamond_{\wdd (X)} U$ is a lattice isomorphism between $\mathcal O(X)$ and $\mathcal O(X^w)$, and whence we have the following proposition.

\begin{proposition}\label{wficationLHC}
	Let $X$ be a $T_0$ space. Then
	\begin{enumerate}[\rm (1)]
		\item $X$ is locally hypercompact if{}f $X^w$ is locally hypercompact.
        \item $X$ is a C-space if{}f $X^w$ is a C-space.
	\end{enumerate}
\end{proposition}

\begin{proposition}\label{wdficationLC} For a $T_0$ space $X$, the following conditions are equivalent:
\begin{enumerate}[\rm (1)]
		\item $X$ is core compact.
        \item $X^w$ is core compact.
        \item $X^w$ is locally compact.
\end{enumerate}
\end{proposition}
\begin{proof} (1) $\Leftrightarrow$ (2): Since $\mathcal O(X)$ and $\mathcal O(X^w)$ are lattice-isomorphic.

(2) $\Rightarrow$ (3): By Theorem \ref{WDwf}, $X^w$ is well-filtered. If $X^w$ is core compact, then $X^w$ is locally compact by Corollary \ref{WfLC=CoreC}.

(3) $\Rightarrow$ (2): Trivial.
\end{proof}

\begin{remark}\label{corecompnotLC} In \cite{Hofmann-Lawson} (see also \cite[Exercise V-5.25]{redbook}) Hofmann and Lawson given a core compact $T_0$ space but not locally compact. By Proposition \ref{wdficationLC}, $X^w$ is locally compact. So the local compactness of $X^w$ does not imply the local compactness of $X$ in general.
\end{remark}

\begin{theorem}\label{SmythWD}  Let $X$ be a $T_0$ space. If $P_S(X)$ is well-filtered determined, then $X$ is well-filtered determined.
\end{theorem}

\begin{proof} Let $A\in\ir_c(X)$, $Y$ a well-filtered space and $f:X\longrightarrow Y$  a continuous mapping. Then $\overline{\xi_X(A)}=\Diamond A\in\ir_c(P_S(X))=\wdd (P_S(X))$ since $P_S(X)$ is well-filtered determined, where $\xi_X : X \longrightarrow P_S(X)$, $x\mapsto\ua x$. Define a mapping $P_S(f): P_S(X)\longrightarrow P_S(Y)$ by
	$$\forall K\in\mk(X),\ P_S(f)(K)=\ua f(K).$$
	
	{Claim 1:} $P_S(f)\circ\xi_X=\xi_Y\circ f$.
	
	For each $ x\in X$, we have
	$$P_S(f)\circ\xi_X(x)=P_S(f)(\ua x)=\ua f(x)=\xi_Y\circ f(x),$$
	that is, the following diagram commutes.
	\begin{equation*}
	\xymatrix{
		X \ar[d]_-{f} \ar[r]^-{\xi_X} &P_S(X)\ar[d]^-{P_S(f)}\\
		Y \ar[r]^-{\xi_Y} &P_S(Y)	}
	\end{equation*}

{Claim 2:} $P_S(f): P_S(X)\longrightarrow P_S(Y)$ is continuous.

Let $V\in\mathcal O(Y)$. We have
$$\begin{array}{lll}
P_S(f)^{-1}(\Box V)&=&\{K\in\mk(X): P_S(f)(K)=\ua f(K)\subseteq V\}\\
&=&\{K\in\mk(X): K\subseteq f^{-1}(V)\}\\
&=&\Box f^{-1}(V),
\end{array}$$
which is open in $P_S(X)$. This implies that $P_S(f)$ is continuous.

By Theorem \ref{Smythwf}, $P_S(Y)$ is well-filtered. Since $P_S(f)$ is continuous and $\Diamond A\in \wdd(P_S(X))$,
there exists a unique $Q\in \mk(Y)$ such that $\overline{P_S(f)(\Diamond A)}=\overline{\{Q\}}$.

 {Claim 3:} $Q$ is supercompact.

 Let $\{U_j:j\in J\}\subseteq\mathcal O(X)$ with $Q\subseteq \bigcup_{j\in J}U_j$, i.e., $Q\in\Box \bigcup_{j\in J}U_j$. Note that $\overline{P_S(f)(\Diamond A)}=\overline{\{\ua f(a):a\in A\}}$, thus $\{\ua f(a):a\in A\}\cap \Box \bigcup_{j\in J}U_j\neq\emptyset$. Then there exists $a_0\in A$ and $j_0\in J$ such that
 $Q\subseteq \ua f(a_0)\subseteq U_{j_0}$.

 Hence, by \cite[Fact 2.2]{Klause-Heckmann}, there exists $y_Q\in Y$ such that $Q=\ua y_Q$.

 {Claim 4:} $\overline{f(A)}=\overline{\{y_Q\}}$.

 Note that $\overline{\{\ua f(a):a\in A\}}=\overline{\{\ua y_Q\}}$. Thus for each $y\in f(A)$, $\ua y\in \overline{\{\ua y_Q\}}$, showing that $\ua y_Q\subseteq \ua y$, i.e.,  $y\in\overline{\{y_Q\}}$. This implies that $f(A)\subseteq \overline{\{y_Q\}}$.
 In addition, since $\ua y_Q\in\overline{\{\ua f(a):a\in A\}}=\Diamond \overline{f(A)}$, $\ua y_Q\cap\overline{f(A)}\neq\emptyset$. This implies that $y_Q\in \overline{f(A)}$. Therefore, $\overline{f(A)}=\overline{\{y_Q\}}$.
 \end{proof}

\section{Conclusion}

In this paper, we introduced and instigated two new classes of subsets in $T_0$ spaces - Rudin sets and $\wdd$ sets lying between the class of all closures of directed subsets and that of irreducible closed subsets, as well as  three new types of spaces - $\mathsf{DC}$ spaces, Rudin spaces and $\wdd$ spaces. Rudin spaces lie between $\wdd$ spaces and $\dc$ spaces, and $\dc$ spaces lie between Rudin spaces and sober spaces. Through such spaces, sober spaces can be factored. More precisely, for a $T_0$ space $X$, it is proved that the following conditions are equivalent: (1) $X$ is sober; (2) $X$ is a $\mathsf{DC}$ $d$-space; (3) $X$ is a well-filtered $\mathsf{DC}$ space; (4) $X$ is a well-filtered Rudin space; and (5) $X$ is a well-filtered $\mathsf{WD}$ space. It is shown that locally hypercompact $T_0$ spaces are $\mathsf{DC}$ spaces, locally compact $T_0$ spaces are Rudin spaces, and core compact $T_0$ spaces are $\wdd$ spaces. As a corollary we have that every core compact well-filtered space is sober, giving a positive answer to Jia-Jung problem \cite{jia-2018}. Using Rudin sets and $\wdd$ sets, we formulate and prove a number of  new characterizations of well-filtered spaces and sober spaces.

Recently, following Keimel and Lawson's method \cite{Keimel-Lawson}, which originated from Wyler's method \cite{Wyler}, Wu, Xi, Xu and Zhao [9] gave a positive answer to the above problem. Following Ershov's method of constructing the $d$-completion of $T_0$ spaces, Shen, Xi, Xu and Zhao have presented a construction of the well-filtered reflection of $T_0$ spaces. In this paper, using $\wdd$ sets, we give a direct approach to well-filtered reflections of $T_0$ spaces, and show that products of well-filtered spaces are well-filtered. Some important properties of well-filtered reflections of $T_0$ spaces are investigated. Comparatively, the technique presented in the paper is not just more direct, but also simpler. Furthermore, it can be also applied to the general $K$-ifications considered  by Keimel and Lawson \cite{Keimel-Lawson}.

Our work shows that $\mathsf{DC}$ spaces, Rudin spaces and $\wdd$ spaces may deserve further investigation. Our study also leads to a number of problems, whose answering will deepen our understanding of the related spaces and structures.

We now close our paper with the following questions about Rudin spaces, $\wdd$ spaces, products of $\wdd$ spaces and well-filtered reflections of products of $T_0$ spaces.

\begin{question}\label{R-Wquestion} Does $\mathsf{RD}(X)=\mathsf{WD}(X)$ hold for ever $T_0$ space $X$?
\end{question}

\begin{question} Is every well-filtered determined space a Rudin space?
\end{question}

\begin{question}\label{WDinfiniteset-prodquestion} Let $X=\prod_{i\in I}X_i$ be the product space of a family $\{X_i: i\in I\}$ of $T_0$ spaces. If each $A_i\subseteq X_i (i\in I)$ is a $\wdd$ set, must  the product set $\prod_{i\in I}A_i$ be a $\wdd$ set of $X$?
\end{question}

\begin{question}\label{WDinfinite-prodquestion} Is the product space of an arbitrary  collection of $\wdd$ spaces well-filtered determined?
\end{question}

\begin{question}\label{infiniteWFication} Does $(\prod\limits_{i\in I}X_i)^w=\prod\limits_{i\in I}X_i^w$ (up to homeomorphism) hold for any family $\{X_i : i\in I\}$ of $T_0$ spaces?
\end{question}

\begin{question}\label{SmythWD} Is the Smyth power space $P_S(X)$ of a well-filtered determined $T_0$ space $X$ again well-filtered determined?
\end{question}


\begin{thebibliography}{99}

\bibitem{Engelking} R. Engelking, General Topology, Polish Scientific Publishers, Warzawa, 1989.
\bibitem{E_2009} M. Ern\'e, Infinite distributive laws versus local connectedness and compactness properties. Topol. Appl.
    156 (2009) 2054每2069.
\bibitem{well}M. Ern\'e, Sober spaces, well-filtration and compactness principles. 2007, http://www.iazd.uni-hannover.de/erne/pre-
    prints/sober.pdf.
\bibitem{E_20181}M. Ern\'e, The strength of prime ideal separation, sobriety, and compactness theorems. Topol. Appl. 241 (2018) 263每290.
\bibitem{E_20182} M. Ern\'e, Categories of Locally Hypercompact Spaces and Quasicontinuous Posets, Applied Categorical Structures. 26 (2018) 823-854.
\bibitem{Esc-Laws-Simp2004} M. Escar\'do, J. Lawson, A. Simpson, Comparing Cartesian closed categories of (core) compactly generated spaces, Topol. Appl. 143 (2004) 105每145.
\bibitem{Hofmann-Lawson} K. Hofmann and J. Lawson, The spectral theory of distributive continuous lattices.
         Trans. of the Amer. Math.l Soc. 246 (1978) 285每310.
\bibitem{redbook} G. Gierz, K. Hofmann, K. Keimel, J. Lawson, M. Mislove, D. Scott, Continuous Lattices and Domains, Encycl. Math. Appl., vol. 93, Cambridge University Press, 2003.
\bibitem{Jean-2013} J. Goubault-Larrecq, Non-Hausdorff topology and Domain Theory, New Mathematical Monographs, vol. 22, Cambridge University Press, 2013.
\bibitem{gcont} G. Gierz, J. Lawson, Generalized continuous and hypercontinuous lattices, Rocky Mt. J. Math. 11 (1981) 271每296.
\bibitem{quasicont} G. Gierz, J. Lawson, A. Stralka, Quasicontinuous posets, Houst. J. Math. 9 (1983) 191每208.
\bibitem{Heckmann} R. Heckmann, An upper power domain construction in terms of strongly compact sets, in: Lecture Notes in Computer Science, vol. 598, Springer, Berlin Heidelberg New York, 1992, pp. 272-293.
\bibitem{Keimel-Lawson} K. Keimel, J. Lawson, $D$-completion and $d$-topology, Ann. Pure Appl. Log. 159 (3) (2009) 292每306.
\bibitem{Klause-Heckmann} R. Heckmann, K. Keimel, Quasicontinuous domains and the Smyth powerdomain, Electronic Notes in Theor. Comp. Sci. 298 (2013) 215-232.
\bibitem{Hofmann-Mislove} K. Hofmann and M. Mislove, Local compactness and continuous lattices, in: Lecture Notes in Mathematics, vol. 871, 1981, pp.125-158.
\bibitem{Kou}H. Kou, $U_k$-admitting dcpo＊s need not be sober, in: Domains and Processes, Semantic Structure on Domain Theory, vol. 1, Kluwer, 2001, pp. 41-50.
\bibitem{isbell}J. Isbell,  Completion of a construction of Johnstone, Proc. Amer. Math. Soci. 85 (1982) 333-334.
\bibitem{jia-2018} X. Jia, Meet-Continuity and Locally Compact Sober Dcpos, PhD thesis, University of Birmingham, 2018.
\bibitem{jia-Jung-2016} X. Jia, A. Jung, A note on coherence of dcpos, Topol. Appl. 209 (2016) 235-238.
\bibitem{johnstone-81} P. Johnstone, Scott is not always sober, in: Continuous Lattices, Lecture Notes in Math., vol. 871, Springer-Verlag, 1981, pp. 282-283.
\bibitem{Lawson-Xi} J. Lawson, X. Xi, Well-filtered spaces, compactness, and the lower topology, preprint.
\bibitem{Rudin} M. Rudin, Directed sets which converge, in: General Topology and Modern Analysis, University of California, Riverside, 1980, Academic Press, 1981, pp. 305每307.
\bibitem{Schalk} A. Schalk, Algebras for Generalized Power Constructions, PhD Thesis, Technische Hochschule Darmstadt, 1993.
\bibitem{Shenchon} C. Shen, X. Xi, X. Xu, D. Zhao, On well-filtered reflections of $T_0$ spaces, Topol. Appl. DOI: 10.1016/j.topol.2019.106869.
\bibitem{wu-xi-xu-zhao-19}G. Wu, X. Xi, X. Xu, D. Zhao, Existence of well-filterification, arXiv:1906.10832 [math.GN](2019).
\bibitem{Wyler} U. Wyler, Dedekind complete posets and Scott topologies, in: Lecture Notes in Mathematics, vol. 871, 1981, pp. 384-389.
\bibitem{Xi-Lawson-2017} X. Xi, J. Lawson, On well-filtered spaces and ordered sets, Topol. Appl. 228 (2017) 139-144.
\bibitem{xi-zhao-MSCS-well-filtered} X. Xi, D. Zhao, Well-filtered spaces and their dcpo models, Math. Struct. Comput. Sci. 27 (2017) 507-515.
\bibitem{xuxizhao} X. Xu, X. Xi, D. Zhao, A complete Heyting algebra whose Scott topology is not sober, arXiv:1903.00615 [math.GN](2019).
\bibitem{ZhaoHo}D. Zhao, W. Ho, On topologies defined by irreducible sets, Journal of Logical and
   Algebraic Methods in Programmin. 84(1) (2015) 185-195.
\end{thebibliography}
\end{document}